\documentclass[oneside, 12pt]{amsart}
\usepackage{amsmath}
\usepackage{amssymb}
\usepackage{amsbsy}
\usepackage{color}
\usepackage[usenames,dvipsnames,x11names,svgnames]{xcolor}
\usepackage[colorlinks=true,linkcolor=NavyBlue,citecolor=DarkGreen]{hyperref}
\usepackage{amsthm}
\usepackage{amscd}
\usepackage[margin=1in]{geometry}
\usepackage{mathrsfs}
\usepackage{dsfont}
\usepackage{todonotes}
\usepackage{mathtools}
\usepackage{enumitem}

\newtheorem{thm}{Theorem}

\newtheorem{prop}[thm]{Proposition}

\newtheorem{lem}[thm]{Lemma}
\newtheorem{cor}[thm]{Corollary}
\theoremstyle{definition}

\newtheorem*{rem}{Remark}

\newcommand{\dt}{\, \textup{d}t}

\newcommand{\be}{\begin{equation*}}
	\newcommand{\ee}{\end{equation*} }

\newcommand{\ba}{\begin{align*}}
	\newcommand{\ea}{\end{align*}}

\newcommand{\ben}{\begin{equation}}
	\newcommand{\een}{\end{equation} }

\newcommand{\bs}{\begin{split}}
	\newcommand{\es}{\end{split}}

\newcommand{\bmu}{\begin{multline*}}
	\newcommand{\emu}{\end{multline*}}

\newcommand{\bmun}{\begin{multline}}
	\newcommand{\emun}{\end{multline}}

\author[M.~V.~Hagen]{Markus Valås Hagen}
\address{Department of Mathematical Sciences, Norwegian University of Science and Technology (NTNU), 7491 Trondheim, Norway} 
\email{markus.v.hagen@ntnu.no}

\thanks{Research supported in part by Grant 334466 of the Research Council of Norway. Some of the work for this article was carried out while the author was in residence at Institut Mittag-Leffler in Djursholm, Sweden during the spring semester of 2024, supported by the Swedish Research Council under grant no. 2021-06594. The author is grateful for the hospitality of IML and the excellent working conditions provided.}


\begin{document}
	
	\title{Sharp conditional moment bounds for products of $L$-functions}
	\maketitle
	\begin{abstract}
		Assuming the Generalized Riemann Hypothesis and the Generalized Ramanujan Conjecture, we determine the order of the $2(k_1,\dots,k_r)$th moment of a product of distinct irreducible $L$-functions on the critical line. As a consequence, we obtain conditional information about the independence of these $L$-functions in the large deviations regime. We also obtain sharp moment bounds for Hurwitz zeta functions with rational parameter, and a certain family of Dedekind zeta functions.
	\end{abstract}
	
	\section{Introduction}
	One of the fundamental challenges in analytic number theory is to understand the behaviour of $L$-functions on the critical line. The moments of the $L$-functions under the scope are central in this study. Recently there has been a lot of progress in our understanding of these moments. For the case of the Riemann zeta function we now know, under the assumption of the Riemann Hypothesis (RH), that
	\begin{equation}\label{sharpZetaMoments}
		\int_{1}^{T} |\zeta(\tfrac12+it)|^{2k} \, \text{d}t \asymp T(\log T)^{k^2},
	\end{equation}
	for all positive real $k$. The lower bound implicit in the result above is unconditional, and follows by the two works \cite{RS2} ($k\geq 1$) and \cite{HeapSound1} ($0<k<1$). Unconditional sharp upper bounds, $$\int_{1}^{T} |\zeta(\tfrac12+it)|^{2k}\dt \ll T(\log T)^{k^2},$$ are more rare, and only known for $k\leq2$ \cite{HeapRadSound1}. Conditionally on RH however, Harper \cite{Harper1} recently proved for all $k>0$ that $$\int_{1}^{T} |\zeta(\tfrac12+it)|^{2k} \, \text{d}t \ll T(\log T)^{k^2},$$ by refining the near optimal bound $\ll T(\log T)^{k^2+\varepsilon}$ due to Soundararajan \cite{Sound1}.
	
	In this article we will prove a similar result to (\ref{sharpZetaMoments}), assuming the Generalized Riemann Hypothesis (GRH) and the Generalized Ramanujan Conjecture (GRC)\footnote{When we say that we assume GRH and GRC, it is implicit that we only assume it for the $L$-functions involved.}. To state our two main theorems, we need to introduce some notation, some of which will be made more explicit in the next section. Let $k_1,\dots,k_r>0$ be real numbers. Let $\pi_1,\dots,\pi_r$ be distinct irreducible cuspidal automorphic representations of $\text{GL}(m_1),\dots,\text{GL}(m_r)$ over $\mathbb{Q}$ (resp.) with unitary central character. To each of these representations $\pi_j$ we consider the associated $L$-functions $L(s,\pi_j)$ (to be defined in the next section). We are interested in the quantity $$I_{k_1,\dots,k_r}(\pi_1,\dots,\pi_r) \coloneq \int_{1}^{T} |L(\tfrac12+it,\pi_1)|^{2k_1}\cdots |L(\tfrac12+it,\pi_r)|^{2k_r} \, \text{d}t.$$ We look at this quantity as a measure of joint distributional properties. Independence between the $L$-functions is expected, and hence we should have $$\frac{1}{T}\int_{1}^{T} \prod_{j=1}^r |L(\tfrac12+it,\pi_j)|^{2k_j} \, \textup{d}t \approx \prod_{j=1}^r \frac{1}{T}\int_{1}^{T}|L(\tfrac12+it,\pi_j)|^{2k_j}\, \textup{d}t .$$
	In this direction, Milinovich--Turnage-Butterbaugh \cite{MilinoButter1} proved the close-to-optimal bound
	\begin{equation}\label{MTBAlmostSharp}
		I_{k_1,\dots,k_r}(\pi_1,\dots,\pi_r) \ll T(\log T)^{k_1^2+\dots+k_r^2+\varepsilon}
	\end{equation}
	for any $\varepsilon>0$, assuming GRH and and an averaged form of GRC (the so-called Hypothesis H of Rudnick and Sarnak \cite{RudSar1}). If we restrict ourselves to only one automorphic $L$-function, sharp bounds for the moments are known in some special cases assuming GRH and GRC. If $\pi$ is assumed to be self-contragredient\footnote{This assumption restricts the class of $L$-functions under consideration quite a lot: in particular it implies that the coefficients $a_{\pi}(n)$ in equation (\ref{logarithmicDerivativeDefinition}) are real.}, in addition to the assumptions above, then Tang and Xiao \cite{TangXiao} have proved $$I_{k}(\pi)\ll T(\log T)^{k^2}.$$ Under the same assumptions, Pi \cite{Pi1} proved the lower bound $$I_{k}(\pi)\gg T(\log T)^{k^2}.$$ Our two main theorems sharpen all of these results, and establish the order of the mixed moments (assuming GRH and GRC). In particular there is no need to assume that the automorphic representations considered are self-contragredient anymore.
	
	\begin{thm}\label{mainThmLowerBound}
		Let $k_1,\dots,k_r>0$ be real numbers. Let $\pi_1,\dots,\pi_r$ be distinct irreducible cuspidal automorphic representations of $\textup{GL}(m_1),\dots,\textup{GL}(m_r)$ over $\mathbb{Q}$ (resp.) with unitary central character. Assume GRH and GRC for all the $L$-functions $L(s,\pi_j)$. Then
		$$I_{k_1,\dots,k_r}(\pi_1,\dots,\pi_r) \gg T(\log T)^{k_1^2+\dots+k_r^2}.$$
	\end{thm}
	\begin{thm}\label{mainThmUpperBound}
		Under the same assumptions as in Theorem \ref{mainThmLowerBound},
		$$I_{k_1,\dots,k_r}(\pi_1,\dots,\pi_r) \ll T(\log T)^{k_1^2+\dots+k_r^2}.$$
	\end{thm}
	\noindent In other words, we establish under GRH and GRC that $$I_{k_1,\dots,k_r}(\pi_1,\dots,\pi_r)\asymp T(\log T)^{k_1^2+\dots+k_r^2}.$$  
	
	\subsection{The statistical behaviour of $L$-functions and independence between them}
	There are many reasons to be interested in the moments of products of $L$-functions. One of them, which we focus on in this article, is the statistical independence of $L$-functions. 
	
	The study of the statistical behaviour of $L$-functions arguably started with Selberg \cite{Sel2}. His celebrated central limit theorem states that $\log|\zeta(1/2+it)|$ has approximately a normal distribution with mean $0$ and variance $\frac12 \log\log|t|$. More specifically 
	\begin{equation}\label{selbergCLT}
		\frac{1}{T}\textup{meas}\left(t \in [1,T] : \frac{\log|\zeta(1/2+it)|}{\sqrt{\frac12 \log\log T}}\geq V\right) \sim \frac{1}{\sqrt{2\pi}}\int_{V}^\infty e^{-x^2/2}\,\text{d}x,
	\end{equation}
	for fixed $V$ and $T$ sufficiently large. Selberg showed that this holds uniformly for $V \leq (\log \log \log T)^{\frac12 - \varepsilon}$. In 2011, Radziwi\l\l\, \cite{Rad1} improved this to $V \ll (\log\log T)^{\frac{1}{10}-\varepsilon}$, by introducing some novel ideas we will utilize in this paper. 
	
	{The asymptotic (\ref{selbergCLT}) does likely not persist to hold uniformly for large values of $V$. However, it is conjectured that (\ref{selbergCLT}) should hold uniformly for $V=o(\sqrt{\log \log T})$. We will refer to anything beyond this range as the ``large devitations range'', which is the range we shall be concerned with in this paper. Here, although the asymptotic is conjectured to fail beyond this range, it is expected that the Gaussian behaviour persists to order for larger $V$. When $V$ is of order $\sqrt{\log\log T}$, it has been conjectured by Radziwi\l\l\, that an asymptotic similar to (\ref{selbergCLT}) holds, but with a constant factor\footnote{{Radziwi\l\l's conjecture is way more precise: if $V=k\sqrt{\log\log T}$, the constant $C_k$ in front is the (conjectured) constant such that $\int_{T}^{2T}|\zeta(\tfrac12+it)|^{2k}\dt\sim C_kT(\log T)^{k^2}$.}} in front of the Gaussian integral. Going beyond $\sqrt{\log\log T}$, Soundararajan \cite{Sound1} proved, Gaussian bounds up to $\sqrt{\log \log T}\log\log\log T$ assuming RH. Unconditionally such bounds are known up to $V\leq 2\sqrt{2\log\log T}$ by Heap--Soundararajan \cite{HeapSound1} and Arguin--Bailey \cite{ArguinBailey1, ArguinBailey2} building on previous results of Jutila \cite{Jutila1}. Here the constant $2\sqrt{2}$ comes directly from the largest moment of $\zeta(s)$ that we know, i.e. the fourth. For even larger $V$, Aymone--Heap--Zhao \cite{AymoneHeapZhao1} investigated a random model for the zeta function. It suggested that zeta behaves more or less Gaussian all the way up to the max, if differing ever so slightly - e.g. for $V\approx \sqrt{\log T}$ their model predicts that the distribution function looks more like $e^{-V^2}$ than $e^{-V^2/2}$.}
	
	The formula (\ref{selbergCLT}) can be generalized to apply not only to more general L-functions, but also to multiple $L$-functions simultaneously. This exhibits an indepedence between them in a statistical sense. Indeed, for a set of primitive $L$-functions $L_1,\dots,L_r$ in the Selberg class, Selberg \cite{Sel1} outlined a proof of their independence under certain assumptions for the coefficients in the Dirichlet series and the Euler product. A proof of this is given by Bombieri--Hejhal \cite[Theorem B]{BombHej1}, assuming GRH or other strong statements about the zeroes of the $L_j$'s. They prove for fixed $V_j$ that 
	\begin{equation}\label{bombieriHejhalCLT}
		\frac{1}{T}\textup{meas}\left(t \in [T,2T] : \frac{\log |L_j(\tfrac12+it)|}{\sqrt{\frac12 \log\log T}}\geq V_j, j=1,\dots,r\right) \sim \prod_{j=1}^r \int_{V_j}^\infty e^{-x^2/2} \,\text{d}x,
	\end{equation}
	for $T$ sufficiently large. {Inoue and Li \cite{InoueLi1} have recently proven\footnote{They work in a certain subclass of the Selberg class, where they assume Hypothesis H, but not the Generalized Ramanujan Conjecture.}, assuming a strong zero density estimate, that this result holds uniformly for $V_j \ll (\log \log T)^{\frac{1}{10}-\varepsilon}$, and, similarly to the individual L-function case, it is expected that this asymptotic holds for $V_j=o(\sqrt{\log\log T})$. Subsequently, we anticipate that Gaussian behavior will continue to hold for some time, and perhaps gradually beginning to deviate slightly in the extreme ranges. To this end, Inoue and Li also proved some Gaussian bounds for (\ref{bombieriHejhalCLT}) in restricted ranges of larger $V_j$, both unconditionally (Theorem 2.2) and conditionally (Theorem 2.3). Their conditional bounds lose sharpness when $V_j\asymp \sqrt{\log \log T}$. More specifically in this range, their bounds for (\ref{bombieriHejhalCLT}) are of the form  $$ \exp\left(-\left(c+o(1)\right)(V_1^2+\dots+V_r^2)\right),$$ for some constant $c$.  In this article, we sharpen up their result. Knowing the moments to order, we are able to determine the order of the distribution function in the range $V_j \asymp \sqrt{\log\log T}$. This is our next theorem.} 
	
	\begin{thm}\label{largeDeviationsCorollary}
		Keep the assumptions from Theorem \ref{mainThmLowerBound} and \ref{mainThmUpperBound}. Let $V_j \asymp \sqrt{\log \log T}$. Then $$\frac{1}{T}\textup{meas}\left( t \in [1,T] : \frac{\log|L(\tfrac12+it,\pi_j)|}{\sqrt{\frac12 \log \log T}}\geq V_j, 1\leq j \leq r\right) \asymp \exp\left(-(1+o(1))\frac{V_1^2+\dots+V_r^2}{2}\right).$$ 
	\end{thm}

	\subsection{Corollaries on other moments}	
	We also obtain two other corollaries of Theorem \ref{mainThmLowerBound} and \ref{mainThmUpperBound}, that are more in the moment spirit. The first one is sharp bounds for moments of Dedekind zeta functions. This removes the $\varepsilon$ in the exponent in a result of Milinovich--Turnage-Butterbaugh \cite[Theorem 1.3.]{MilinoButter1}. Such mean value results can be applied to study for example the  coefficients of Dedekind zeta functions in small intervals - see \cite[Theorem 1.4.]{MilinoButter1}.
	
	\begin{cor}\label{dedekindCorollary}
		Let $K$ be a finite solvable Galois extension of $\mathbb{Q}$. Assume RH for $\zeta_K(s)$. Then for any $k>0$, $$\int_{1}^{T} |\zeta_K(\tfrac12+it)|^{2k} \dt \asymp T(\log T)^{k^2[K:\mathbb{Q}]}.$$
	\end{cor}
	
	\begin{proof}
		Following the proof of \cite[Corollary 1.2]{MilinoButter1} and the references therein, this is an immediate consequence of Theorem \ref{mainThmLowerBound} and Theorem \ref{mainThmUpperBound}. 
	\end{proof}
	
	\begin{rem}
		We believe that the assumption of solvability of the Galois group $\textup{Gal}(K/\mathbb{Q})$ can be relaxed in the previous corollary. However with this assumption, Corollary \ref{dedekindCorollary} is a direct consequence of our main theorems, because $\zeta_K$ admits a factorization into irreducible automorphic $L$-functions (which all satisfy GRC) in this case. Working more directly with the Dedekind zeta functions, one should be able to modify the proof to also holds in the non-solvability case. We refer the reader to Section 5 of \cite{MilinoButter1} for more discussion around this.
	\end{rem}
	
	The final corollary is inspired by a paper of Sahay \cite{Sahay1} on the moments of the Hurwitz zeta function with rational parameter. The Hurwitz zeta function $\zeta(s,\alpha)$, with any parameter $0<\alpha\leq 1$ is defined for $\text{Re}(s)>1$ by $$\zeta(s,\alpha)=\sum_{n=0}^\infty \frac{1}{(n+\alpha)^s},$$ and can be extended to a meromorphic function with a simple pole in $s=1$. One of the reasons that these functions are interesting (when $\alpha$ is rational), is that RH fails for more or less all of them, but they do still admit many similarities to $L$-functions for which we believe RH to be true. One of these similarities are the moments.
	
	Let $\alpha=a/q, \gcd(a,q)=1, 1\leq a \leq q$. In \cite{Sahay1}, Sahay conjectured that for any $k>0$, $$\int_{T}^{2T} |\zeta(\tfrac12+it,\alpha)|^{2k}\dt \sim c_k(\alpha)T(\log T)^{k^2},$$ for some constant $c_k(\alpha)$. This is known unconditionally in the $k=1,2$. The case $k=1$ is due to Rane \cite{Rane1}, and also holds for irrational $\alpha$. The case $k=2$ was first proved by Andersson, and can be found in an unpublished section of his thesis \cite[p. 71-72.]{Andersson}. Sahay also gave a proof in \cite{Sahay1}. Furthermore, conditionally on RH for Dirichlet $L$-functions, he proved for all positive integers $k$ that $$T(\log T)^{k^2} \ll \int_{T}^{2T} |\zeta(\tfrac12+it,\alpha)|^{2k}\dt \ll T(\log T)^{k^2+\varepsilon},$$ for any $\varepsilon>0$. There is a very natural reason to assume RH for Dirichlet $L$-functions, namely that 
	\begin{equation}\label{hurwitzZetaDirichletLinearCombs}
		\zeta(s,\alpha)=\frac{q^s}{\varphi(q)}\sum_{\chi} \overline{\chi(a)}L(s,\chi).
	\end{equation}
	Here the sum runs over all Dirichlet characters modulo $q$.
	
	{We improve Sahay's result in two ways:  the $\varepsilon$ in the exponent is removed, and the order of the moments is determined for any \emph{real} $k>1/2$. Using the method of Heap and Soundararajan \cite{HeapSound1} it seems likely that one could extend this all the way down to $k>0$, but this would require us to deviate from the proof strategy that gives a rather simple proof of Corollary \ref{hurwitzCorollary}.}
	
	\begin{cor}\label{hurwitzCorollary}
		Let $\alpha=a/q, \gcd(a,q)=1, 1\leq a \leq q$. Assume RH for all Dirichlet $L$-functions modulo $q$. Then for all $k>1/2$, we have $$T(\log T)^{k^2}\ll\int_{1}^{T} |\zeta(\tfrac12+it,\alpha)|^{2k} \,\textup{d}t \ll T(\log T)^{k^2}.$$
		The upper bound also holds for $k>0$.
	\end{cor}

	\subsection{Proof strategy}
	The strategy to prove Theorem \ref{mainThmLowerBound} is due to Heath--Brown \cite{Heath-Brown1} and dates back to the 1980's. He used this method to prove sharp lower bounds for the $2k$th moments of the Riemann zeta function\footnote{He also used it to prove some sharp upper bounds.}. The idea is roughly as follows. For $\text{Re}(s)>1$, we have $$L^{k_1}(s,\pi_1)\cdots L^{k_r}(s,\pi_r)=\sum_{n=1}^\infty \frac{\mathbf{h}_{k_1,\dots,k_r}(n)}{n^s},$$ for some coefficients $\mathbf{h}_{k_1,\dots,k_r}(n)$ that can be explicitly written down. A trivial application of the triangle inequality tells us that 
	\begin{align*}
		\int_1^{T} \left|\sum_{n\leq N} \frac{\mathbf{h}_{k_1,\dots,k_r}(n)}{n^{1/2+it}}\right|^2\, \text{d}t 
		&\ll \int_1^{T} |L(\tfrac12+it,\pi_1)|^{2k_1}\cdots |L(\tfrac12+it,\pi_r)|^{2k_r}\,\text{d}t \\ 
		&+ \int_1^{T} \left|L^{k_1}(\tfrac12+it,\pi_1)\cdots L^{k_r}(\tfrac12+it,\pi_r) - \sum_{n\leq N} \frac{\mathbf{h}_{k_1,\dots,k_r}(n)}{n^{1/2+it}} \right|^2\dt.
	\end{align*}
	On the right hand side of the inequality, the first term is the quantity we want to bound from below. On the left hand side, we have something that we can compute using the Montgomery--Vaughan mean value theorem. Thus we are left with estimating the last term on the right hand side of the inequality. Using certain convexity theorems we are able to reduce this computation to the same integral but evaluated in $\tfrac54+it$ instead of $\tfrac12+it$. On $\text{Re}(s)=5/4$, the Dirichlet series definition for $L^{k_1}\cdots L^{k_r}$ is valid, and so the computation of the integral quickly reduces to a mean-value computation that again can be computed by using Montgomery--Vaughan. The main point of this last maneuvre is that the integrand on $\text{Re}(s)=5/4$ is given by the tail of an absolutely convergent Dirichlet series, and its mean value squared is clearly very small. Our execution of this method is more or less identical to that of Heath-Brown.

	The strategy for proving Theorem \ref{mainThmUpperBound} is more recent, and is an accumulation of methods of Soundararajan, Radziwi\l\l\, and Harper, which we quickly survey. The first contribution due to Soundararajan appears in \cite{Sound1}. There he shows on RH that $\log|\zeta(1/2+it)|$ can be bounded above by a sum over primes, with no contribution from a sum over non-trivial zeroes. Specifically, this lets us bound $|\zeta(1/2+it)|$ by an Euler product of suitable length. Chandee \cite{Chandee1} generalized Soundararajan's result to more general $L$-functions. Her inequality will be our point of departure for proving Theorem \ref{mainThmUpperBound}. 
	
	The next idea we will utilize is due to Radziwi\l\l\, and first appeared in \cite{Rad1}. Roughly speaking, it allows us to compute moments of Euler products $\mathcal{P}$ on $[T,2T]$, under the assumption that we can restrict to certain subsets of $[T,2T]$ where $\mathcal{P}$ is of \emph{typical} size in a statistical sense (not deviating too much from its mean). Let us decompose $[T,2T]$ as $\mathscr{G} \cup \mathscr{B}$, where $\mathscr{G}$ is the set where $|\mathcal{P}|$ is of typical size, and $\mathscr{B}$ is the set where $|\mathcal{P}|$ is large. To compute an upper bound for $I_{k_1,\dots,k_r}(\pi_1,\dots,\pi_r)$ we first apply Chandee's inequality to go from $\prod_{j=1}^r L(\frac12+it,\pi_j)^{2k_j}$ to an Euler product $\mathcal{P}$. The integral over $\mathscr{G}$ quickly reduces to a mean-value computation that is admissible for Montgomery--Vaughan. The remaining set $\mathscr{B}$ will be so small in measure that the integral over $\mathscr{B}$ will not contribute anything more than the main term coming from the integral over $\mathscr{G}$. The main point of this method is that when we integrate over $\mathscr{G}$, instead of the whole range $[T,2T]$, we can let the length of $\mathcal{P}$ be significantly longer. We refer to \cite[p. 3--4]{Heap1} for an illustrative example of this idea.
	
	The idea described in the paragraph above will almost give sharp upper bounds. To get our desired sharp upper bounds we have to push the length of the Euler product to be slightly longer. The final idea we need appears in both \cite{Harper1} and \cite{RS1}. It revolves around splitting up our Euler product into many smaller Euler products whose variance (i.e. deviation from the mean) gets progressively smaller. Applying the idea of Radziwi\l\l\, locally to each of the Euler products, we can choose the original Euler product to be sufficiently long. Harper successfully applied this approach to exhibit conditional sharp upper bounds for the $2k$th moments of $\zeta(\tfrac12+it)$, and several authors have used the recipe outlined above to establish sharp moments bounds since.
	
	We find that many of the ideas mentioned above are applicable to our setting, without any significant modification. There are several ways to execute Harper's method in our setting, and we will use one that is slightly different from Harper's original approach, but aligns more with the approach of Heap \cite{Heap1}. 	
	
	The reader may have observed that the result $$I_{k_1,\dots,k_r}(\pi_1,\dots,\pi_r) \ll T(\log T)^{k_1^2+\dots+k_r^2+\varepsilon}$$ of Milinovich and Turnage-Butterbaugh does not assume GRC, but Hypothesis H instead, which is a weaker assumption. We believe that this could be done in this setting as well, at the cost of more unpleasant combinatorics in our proof. 
	
	It remains an interesting question how much of our work that could be recovered unconditionally, especially the lower moment bound, because the analogous bounds for $\zeta(s)$ are unconditional. We believe that it is possible to establish the order of lower moments if we restrict ourselves to low-degree $L$-functions. A natural starting point for such an investigation could be the moments $$\int_{1}^T |\zeta(\tfrac12+it)|^{2k_1}|L(\tfrac12+it,\chi)|^{2k_2}\dt,$$ with $0<k_1,k_2\leq 2$. 
	
	Another interesting question to ask is to what extent Theorem \ref{mainThmLowerBound} and \ref{mainThmUpperBound} can be made uniform in $k_1,\dots,k_r$. This would be desireable as uniformity in certain ranges of $k_1,\dots,k_r$, would let us vary $V_1,\dots,V_r$ in some window in Theorem \ref{largeDeviationsCorollary}, instead of keeping them fixed.
	
	The rest of the paper is divided into five more sections. In Section 2 we survey the theory of automorphic $L$-functions, and prove some estimates on the coefficients of these $L$-functions that we will need later. In Section 3 we prove that $I_{k_1,\dots,k_r}(\pi_1,\dots,\pi_r) \gg T(\log T)^{k_1^2+\dots+k_r^2}$, before we in Section 4 prove the opposite $I_{k_1,\dots,k_r}(\pi_1,\dots,\pi_r)\ll T(\log T)^{k_1^2+\dots+k_r^2}$. Finally, we prove Theorem \ref{largeDeviationsCorollary} as well as Corollary \ref{hurwitzCorollary} in the last section.
	
	\section*{Acknowledgements}
	The author is very grateful to his supervisors Winston Heap and Kristian Seip for a lot of insightful discussions and advice, as well as encouragement. The author would also like to thank Junxian Li and Anurag Sahay for taking time to read parts of the manuscript, as well as a comment from Adam Harper that lead to an improvement in Corollary \ref{hurwitzCorollary}.
	
	\section{Background on automorphic $L$-functions and some estimates}
	We will start this section by surveying the theory of automorphic $L$-functions, following Rudnick--Sarnak \cite{RudSar1} and Iwaniec--Kowalski \cite{IwaniecKowalski1}. Afterwards we will derive some estimates we will need later. To this end, let $\pi$ be an irreducible cuspidal automorphic representation of $\text{GL}(m)$ over $\mathbb{Q}$, with unitary central character. For each prime $p$, we define a local Euler product 
	$$L_p(s,\pi)\coloneq\prod_{j=1}^m \left(1-\frac{\alpha_{\pi}(j,p)}{p^s}\right)^{-1}.$$ 
	We shall not be concerned about the origin nor the definition of those $\alpha_{\pi}(j,p)$, but mention that they are associated to $\pi$. For all, but finitely many $p$, the Generalized Ramanujan Conjecture (GRC) predicts that 
	\begin{equation}\label{ramPetConj}
		|\alpha_\pi(j,p)|=1.
	\end{equation}
	For the remaining $p$, the conjecture predicts that $|\alpha(j,p)|\leq 1$ (see e.g. \cite[p. 95]{IwaniecKowalski1}). In \cite{RudSar1}, Rudnick and Sarnak proved that
	\begin{equation}\label{RudSarRamPet}
		|\alpha_\pi(j,p)|\leq p^{\frac12-\frac{1}{m^2+1}}.
	\end{equation}
	From these local factors we define the $L$-function associated to $\pi$ by $$L(s,\pi)\coloneq\prod_{p<\infty} L_p(s,\pi).$$ By the theory of Rankin--Selberg functions, we have absolute convergence of $L(s,\pi)$ for $\text{Re}(s)>1$ (see \cite[Chapter 5.12]{IwaniecKowalski1} and the references therein). We say that this $L$-function is of degree $m$ (according to the number of factors in the local factors $L_{p}(s,\pi)$)
	
	To make this precise, we need some defintions. Let\footnote{The $\pi$ in the formula below is the numerical constant, and should not be confused with the representation $\pi$.} $$\Gamma_{\mathbb{R}}(s)\coloneq\Gamma(s/2)\pi^{-s/2},$$ and $$L_{\infty}(s,\pi)\coloneq\prod_{j=1}^m \Gamma_{\mathbb{R}}(s+\mu_\pi(j)).$$ Here, $\mu_{\pi}(j)$ are constants depending on $\pi$. They satisfy\footnote{This is not the best bound. See \cite[Ch. 5.12]{IwaniecKowalski1}.} $$\text{Re }{\mu_\pi(j)}>-\tfrac12.$$ From the representation $\pi$, we can associate a new representation $\widetilde{\pi}$, called its contragradient. This is again a new irreducible cuspidal automorphic representation, and locally it is equivalent to the representation $\overline{\pi_{p}}$ for $p\leq \infty$, which means in particular that $$\{\alpha_{\widetilde{\pi}}(j,p)\}=\{\overline{\alpha_{\pi}(k,p)}\} \qquad \{\mu_{\widetilde{\pi}}(j)\} = \{\overline{\mu_\pi(j)}\}.$$ 
	
	Let now $\Phi(s,\pi)\coloneq L(s,\pi)L_{\infty}(s,\pi)$ be the completed $L$-function of $L(s,\pi)$. Then it has been proved (see Rudnick--Sarnak \cite{RudSar1} and the references therein) that 
	\begin{enumerate}
		\item $\Phi(s,\pi)$ extends to an entire function of order one, in all cases except the case where $L(s,\pi)$ is the Riemann zeta function.
		\item We have the following functional equation $$\Phi(s,\pi)=N^{1/2-s}\epsilon_{\pi}\Phi(1-s,\widetilde{\pi})$$ where $N>1$ is some integer and $\epsilon_{\pi}$ has absolute value $1$.
		\item $\Phi(s,\pi)$ is bounded in vertical strips.
	\end{enumerate}
	
	{Throughout this paper we will assume that $L(s,\pi)$ has central value $1/2$, so that in the case $\Phi(s,\pi)$ is not entire, it is truly the Riemann zeta function. Technically speaking, without this assumption, $L(s,|\cdot|^{it})=\zeta(s+it)$ is also a completely valid function in this framework. As far as the author understands, this assumption seem to be implicit in some of the literature, see e.g. Rudnick--Sarnak \cite{RudSar1}.}
	
	The most famous conjecture for these $L$-functions is GRH. It states that all the zeroes of $\rho_{\pi}$ of $\Phi(s,\pi)$ (equivalently all the non-trivial zeroes of $L(s,\pi)$) have $\text{Re}(\rho_\pi)=1/2$. A consequence of of GRH is the Generalized Lindelöf Hypothesis (GLH) (see e.g. \cite[Corollary 5.20]{IwaniecKowalski1}). It states (in the $t$-aspect) that $$L(\tfrac12+it,\pi) \ll t^\varepsilon$$ for any $\varepsilon>0$, where the implicit constant depend on $\pi$ and $\varepsilon$. 
	
	\subsection{The coefficients of $L(s,\pi)$}
	{For $\text{Re}(s)>1$ we have $$L(s,\pi)=\sum_{n \geq 1} \frac{A_{\pi}(n)}{n^s}=\prod_p \prod_{j=1}^m \left(1-\frac{\alpha_{\pi}(j,p)}{p^s}\right)^{-1}.$$Logarithmically differentiating the Euler product for $L(s,\pi)$ above, we obtain a convergent Dirichlet series
	\begin{equation}\label{logarithmicDerivativeDefinition}
		\frac{\text{d}}{\text{d}s} \log L(s,\pi) = -\sum_{n\geq 1} \frac{\Lambda_{\pi}(n)}{n^s}\coloneq-\sum_{n\geq 2}\frac{\Lambda(n)a_{\pi}(n)}{n^s}
	\end{equation}
	for $\text{Re}(s)>1$, where $a_{\pi}(n)$ is defined implicitly. Here, $\Lambda_\pi(n)$ is supported on prime powers, and we have $$\Lambda_\pi(p^\ell) = \sum_{j=1}^{m} \alpha_{\pi}(j,p)^\ell \log p = \Lambda(p^\ell)\sum_{j=1}^{m} \alpha_{\pi}(j,p)^\ell.$$ }
	
	The behaviour of the coefficients $\Lambda_\pi(n)$ will play a key role in the proof of Theorem \ref{mainThmUpperBound}. In general, not much is known about their nature. Unconditionally, we start by observing that (\ref{RudSarRamPet}) implies 
	\begin{equation} \label{unconditionalVonMangoldtBound}
		|\Lambda_{\pi}(n)|\leq m\Lambda(n)n^{\frac12 - \frac{1}{m^2+1}}.
	\end{equation}
	Assuming GRC however, we have the much stronger 
	\begin{equation}\label{conditionalVonMangoldtBound}
		|\Lambda_{\pi}(n)|\leq m\Lambda(n).
	\end{equation}

	As we are working with several $L$-functions simultaneously, we have to understand the correlation between them. Again, we lack solid understanding of this in general, apart from some low-dimensional cases. However, since we are assuming GRC\footnote{ Selberg's orthogonality conjecture follows from less strong assumptions as well, like Hypothesis H.} we have Selberg's orthogonality conjecture at our disposal, which tells us a lot about the correlation. It was first formulated by Selberg in \cite{Sel2}, albeit in a different context. Again, we refer the reader to \cite[Theorem 2.2]{MilinoButter1} and the references therein for a proof of this. \newline
	
	\noindent \textbf{Selberg's orthogonality conjecture.} Let $\pi$ and $\pi'$ be irreducible cuspidal automorphic representations of $\text{GL}(m)$ and $\text{GL}(m')$ over $\mathbb{Q}$ respectively, both with unitary central character. If $x\geq 3$, we have 
	$$\sum_{p \leq x} \frac{a_{\pi}(p)\overline{a_{\pi'}(p)}}{p} = \left\{\begin{matrix} \log\log x + O(1), & \text{if }\pi\cong \pi' \\ O(1), & \text{ if }\pi \not\cong \pi'.   \end{matrix}\right.$$
	
	Selberg's orthogonality conjecture holds whenever $\max(m,m')\leq 4$ (see \cite[Theorem 2.2]{MilinoButter1} and the references therein.) If we have distinct irreducible cuspidal automorphic representations $\pi_1,\dots,\pi_r$ with unitary central characters, and positive real numbers $k_1,\dots,k_r$, we notice that the conjecture implies $$\sum_{p \leq x} \frac{|k_1a_{\pi_1}(p)+\dots+k_ra_{\pi_r}(p)|^2}{p} =(k_1^2+\dots+k_r^2)\log\log x + O(1)$$

	\subsection{The coefficients of $L^{k_1}(s,\pi_1)\cdots L^{k_r}(s,\pi_r)$}
	For the proof of Theorem \ref{mainThmLowerBound}, we will need to understand the coefficients of 
	$$L^{k_1}(s,\pi_1)\cdots L^{k_r}(s,\pi_r)$$ for $k_i \in \mathbb{R}$ and $\text{Re}(s)>1$. We start by writing 
	$$L^{k_1}(s,\pi_1)\cdots L^{k_r}(s,\pi_r)=\exp\left(k_1\log L(s,\pi_1)+\dots+k_r\log L(s,\pi_r)\right).$$ We choose a branch of the logarithm so that $$\log L(s,\pi_i)=\sum_p \sum_{j=1}^{m_i}\sum_{\ell=1}^\infty \frac{\alpha_{\pi_i}(j,p)^\ell}{\ell p^{\ell s}}.$$ This, in turn defines a branch of $L^{k_1}(s,\pi_1)\cdots L^{k_r}(s,\pi_r)$ by $$L^{k_1}(s,\pi_1)\cdots L^{k_r}(s,\pi_r) = \prod_{i=1}^r\prod_{p} \prod_{j=1}^{m_i} \left(1-\frac{\alpha_{\pi_i}(j,p)}{p^s}\right)^{-k_i}.$$ For $|z|<1$, we have the following Taylor series expansion $$(1-z)^{-k} =\sum_{\ell=0}^\infty \frac{\Gamma(k+\ell)}{\Gamma(k)\ell!}z^\ell.$$ Observe that by (\ref{RudSarRamPet}), $$\left|\frac{\alpha_{\pi_i}(j,p)}{p^{s}}\right|\leq p^{1/2-\frac{1}{m^2+1}-\text{Re}(s)}<1$$ for $\text{Re}(s)>1$. Let us define $$d_k(p^\ell)\coloneq \frac{\Gamma(k+\ell)}{\Gamma(k)\ell!}$$ and extend multiplicatively. Furthermore we define 
	$$h_{k_i}(p^\ell) \coloneq \sum_{\ell_1+\dots+\ell_{m_i} =\ell, \ell_\rho \geq 0} d_{k_i}(p^{\ell_1})\alpha_{\pi_i}(1,p)^{\ell_1}\cdots d_{k_i}(p^{\ell_{m_i}}) \alpha_{\pi_i}(m_i,p)^{\ell_{m_i}},$$
	and extend multiplicatively. Notice in particular that $h_{k_i}(p)=k_ia_{\pi_i}(p)$. Let us furthermore define 
	$$\mathbf{h}_{k_1,\dots,k_r}(n) \coloneq \sum_{n_1\cdots n_r = n, n_i \geq 1} h_{k_1}(n_1)\cdots h_{k_r}(n_r).$$ Again we notice that $\mathbf{h}_{k_1,\dots,k_r}(p)=k_1a_{\pi_1}(p)+\dots+k_ra_{\pi_r}(p)$.
	The formulas above yield the following expression
	\begin{align*}
		L^{k_1}(s,\pi_1)\cdots L^{k_r}(s,\pi_r) &= \prod_{i=1}^r\prod_{p} \left(\sum_{\ell =0}^\infty \frac{h_{k_i}(p^\ell)}{p^{\ell s}}\right) \\
		&= \prod_{i=1}^r \sum_{n=1}^\infty \frac{h_{k_i}(n)}{n^s} \\
		&= \sum_{n=1}^\infty \frac{\mathbf{h}_{k_1,\dots,k_r}(n)}{n^s}
	\end{align*}
	for $\text{Re}(s)>1$. 
	
	We now gather a few lemmata for partial sums of those coefficients.
	
	\begin{lem}\label{h_kHigherCoeffsCanBeIgnored}
		Let $\pi_1,\dots,\pi_r$ be distinct irreducible cuspidal automorphic representations of $\textup{GL}(m_1),\dots,\textup{GL}(m_r)$ over $\mathbb{Q}$ respectively, with unitary central character. Let $k_1,\dots,k_r>0$. Assume GRC. Then $$|\mathbf{h}_{k_1,\dots,k_r}(p^\ell)|\ll p^{\ell\varepsilon},$$ for any $\varepsilon>0$. Consequently, $$\sum_{p,\ell\geq 2} \frac{|\mathbf{h}_{k_1,\dots,k_r}(p^\ell)|^2}{p^\ell}<\infty.$$
	\end{lem}
	
	\begin{proof}
		This follows from standard divisor coefficient facts (see e.g. \cite[Lemma 1]{Heath-Brown1}):
		\begin{align*}
			\mathbf{h}_{k_1,\dots,k_r}(p^\ell) &\ll \sum_{\substack{\ell_1+\dots+\ell_r=\ell, \\ \ell_j\geq 0}} |h_{k_1}(p^{\ell_1})|\cdots |h_{k_r}(p^{\ell_r})| \\
			&\ll \sum_{\substack{\ell_1+\dots+\ell_r=\ell, \\ \ell_j\geq 0}} \Big{(} \prod_{\eta=1}^r \Big{(}\sum_{\substack{r_1+\dots+r_{m_\eta}=\ell_\eta,\\ r_j \geq 0}} |d_{k_\eta}(p^{r_1})|\cdots |d_{k_\eta}(p^{r_{m_\eta}})|\Big{)}\Big{)} \\
			&\ll \sum_{\substack{\ell_1+\dots+\ell_r=\ell, \\ \ell_j\geq 0}} d_{k_1m_1}(p^{\ell_1})\cdots d_{k_rm_r}(p^{\ell_r}) \\
			&\ll d_{\max_{1 \leq i \leq r}\{rk_im_i\}}(p^\ell) \ll p^{\varepsilon \ell}.
		\end{align*}
	\end{proof}
	
	\begin{lem}\label{h_kCoeffsAsymps}
		Let $\pi_1,\dots,\pi_r$ be distinct irreducible cuspidal automorphic representations of $\textup{GL}(m_1),\dots,\textup{GL}(m_r)$ over $\mathbb{Q}$ respectively, with unitary central character. Let $k_1,\dots,k_r>0$. Then there exists a constant $C_{k_1,\dots,k_r}>0$ such that
		\begin{align*}
			\textup{(i)} &\enspace  \sum_{n \leq N} \frac{|\mathbf{h}_{k_1,\dots,k_r}(n)|^2}{n^{2\sigma}} \asymp \left(\sigma-\frac12\right)^{-(k_1^2+\dots+k_r^2)} \enspace \textup{uniformly for } \frac{1}{2}+\frac{C_{k_1,\dots,k_r}}{\log N}\leq \sigma \leq 1, \\
			\textup{(ii)} &\enspace \sum_{n \leq N} \frac{|\mathbf{h}_{k_1,\dots,k_r}(n)|^2}{n} \asymp (\log N)^{k_1^2+\dots+k_r^2}.
		\end{align*}
	\end{lem}
	
	\begin{proof}
		The second point follows readily from the first, so we will only prove the first point. {The statement is clear when $\sigma$ is bounded away from $\frac12$ by an absolute constant, because the sum is bounded in that case. Thus it suffices to consider the case when $\sigma$ is close to $\frac12$. Let us by $S(Y)$ denote the set of positive integers such that $p \mid n \implies p \leq Y$, where $p$ is some prime number. In the proceeding, we find that we need asymptotics for the following function $E_1(x) \coloneq \int_{x}^\infty e^{-t}/t \, \text{d}t$. We have $$E_1(x)=-\gamma-\log x - \sum_{k=1}^\infty \frac{(-x)^k}{k!k}$$ for $x>0$. Hence by this asymptotic for small $x$, integration by parts and Selberg orthogonality, we have 
			\begin{align*}
				\sum_{p} \frac{|\mathbf{h}_{k_1,\dots,k_r}(p)|^2}{p^{2\sigma}} &= (1+o(1))(k_1^2+\dots+k_r^2)\int_{3}^\infty \frac{\textup{d}t}{t^{2\sigma}\log t} \\
				&= (1+o(1))(k_1^2+\dots+k_r^2)E_1((2\sigma-1)\log 3) \\
				&= (1+o(1))\log\left((\sigma-\tfrac12)^{-(k_1^2+\dots+k_r^2)}\right).
		\end{align*} }Thus
		\begin{align*}
			\sum_{p \leq Y} \frac{|\mathbf{h}_{k_1,\dots,k_r}(p)|^2}{p^{2\sigma}} &= \sum_{p} \frac{|\mathbf{h}_{k_1,\dots,k_r}(p)|^2}{p^{2\sigma}} - \sum_{p > Y} \frac{|\mathbf{h}_{k_1,\dots,k_r}(p)|^2}{p^{2\sigma}} \\ 
			&= \log\left((\sigma-\tfrac12)^{-(k_1^2+\dots+k_r^2)}\right) + O\left((k_1^2+\dots+k_r^2)E_1((2\sigma-1)\log Y)\right),
		\end{align*}
		for any $Y>1$. We want to ensure that the error term here is $O(1)$, and for that we have to ensure that $(k_1^2+\dots+k_r^2)(2\sigma-1)\log Y$ is bounded below by some positive constant $C$. In other words, we need $\sigma \geq \frac12 + \frac{C}{(k_1^2+\dots+k_r^2)\log Y}$. The analysis above does not impose any condition on this constant - only that it has to be positive and absolute. We will choose it towards the end of the proof. Thus we conclude under this condition that
		\begin{equation}\label{hkCoeffsAsymptoticWithO(1)ErrorTerm}
			\sum_{p \leq Y} \frac{|\mathbf{h}_{k_1,\dots,k_r}(p)|^2}{p^{2\sigma}} = \log\left((\sigma-\tfrac12)^{-(k_1^2+\dots+k_r^2)}\right) + O\left(1\right).
		\end{equation}
		
		We first deal with the upper bound implicit in (i). Bounding by $N$-smooth numbers we have
		$$\sum_{n \leq N} \frac{|\mathbf{h}_{k_1,\dots,k_r}(n)|^2}{n^{2\sigma}} \leq \sum_{n \in S(N)} \frac{|\mathbf{h}_{k_1,\dots,k_r}(n)|^2}{n^{2\sigma}} = \prod_{p \leq N} \left(1+\frac{|\mathbf{h}_{k_1,\dots,k_r}(p)|^2}{p^{2\sigma}}+\frac{|\mathbf{h}_{k_1,\dots,k_r}(p^2)|^2}{p^{4\sigma}}+\dots\right).$$ 
		Because $\ell\geq 2$, $$\sum_{p,\ell \geq 2} \frac{|\mathbf{h}_{k_1,\dots,k_r}(p^\ell)|^2}{p^{2\ell\sigma}}$$ converges\footnote{Observe that this sum converges for any $\sigma\geq 1/2$, in particular independently of $N$.} by Lemma \ref{h_kHigherCoeffsCanBeIgnored}, and so on applying $1+x \leq \exp(x)$ and (\ref{hkCoeffsAsymptoticWithO(1)ErrorTerm}) we get the desired upper bound
		$$\sum_{n \leq N} \frac{|\mathbf{h}_{k_1,\dots,k_r}(n)|^2}{n^{2\sigma}} \ll \exp\left(\sum_{p \leq N} \frac{|\mathbf{h}_{k_1,\dots,k_r}(p)|^2}{p^{2\sigma}} \right)\asymp\left(\sigma-\frac12\right)^{-(k_1^2+\dots+k_r^2)},$$ assuming $\sigma \geq \frac12 + \frac{C_1}{(k_1^2+\dots+k_r^2)\log N}$ for some unspecified constant $C_1>0$. Observe again that the value of this constant does not matter at this point, as long as we keep $\frac{C_1}{(k_1^2+\dots+k_r^2)}$ bounded away from zero. 
		
		Let us now turn to the lower bound. To this end, let $Y=N^{\theta}$ with $0<\theta<1$, and assume $\sigma \geq \frac12 + \frac{C_2}{(k_1^2+\dots+k_r^2)\log Y}$. We have
		\begin{equation}\label{rankinLowerBoundStart}
			\sum_{n \leq N} \frac{|\mathbf{h}_{k_1,\dots,k_r}(n)|^2}{n^{2\sigma}} \geq  \sum_{n \in S(Y)} \frac{|\mathbf{h}_{k_1,\dots,k_r}(n)|^2}{n^{2\sigma}} - \sum_{\substack{n >N, \\ n \in S(Y)}} \frac{|\mathbf{h}_{k_1,\dots,k_r}(n)|^2}{n^{2\sigma}}.
		\end{equation}
		Following the proof of the upper bound, we have that the first term is $$\geq B\left(\sigma-\frac12\right)^{-(k_1^2+\dots+k_r^2)}$$ for some absolute fixed constant $B$. For the error term we apply Rankin's trick and approximate $p^\delta$ by its Taylor series. Let $\delta=1/\log Y$. Then 
		\begin{align*}
			&\sum_{\substack{n >N, \\ n \in S(Y)}} \frac{|\mathbf{h}_{k_1,\dots,k_r}(n)|^2}{n^{2\sigma}} \\
			\leq &\frac{1}{N^\delta}\sum_{n \in S(Y)} \frac{|\mathbf{h}_{k_1,\dots,k_r}(n)|^2}{n^{2\sigma-\delta}} \\
			\ll &\frac{1}{N^{\delta}} \prod_{p \leq Y} \left(1+\frac{|\mathbf{h}_{k_1,\dots,k_r}(p)|^2}{p^{2\sigma-\delta}}\right) \\
			\ll &\exp\left(-\delta\log N + \sum_{p \leq Y} \frac{|\mathbf{h}_{k_1,\dots,k_r}(p)|^2}{p^{2\sigma}} + +O\left(\delta\sum_{p \leq Y}\frac{|\mathbf{h}_{k_1,\dots,k_r}(p)|^2\log p}{p^{2\sigma}} \right)\right) \\
			\ll  &\left(\sigma-\frac12\right)^{-(k_1^2+\dots+k_r^2)}\exp\left(-\delta \log N + O\left(\delta (k_1^2+\dots+k_r^2)\frac{Y^{1-2\sigma}}{2\sigma-1}\right)\right).
		\end{align*}
		Here we have used that $\delta\log p \leq 1$. Now we choose $\theta=1/A$ for $A$ some large constant to be determined later. Then, recalling $\sigma \geq \frac12 + \frac{C_2}{(k_1^2+\dots+k_r^2)\log Y}$, the latter factor in the last displayed line above is $$\leq \exp(-A+ (k_1^2+\dots+k_r^2)^2\frac{1}{C_2}),$$ and thus taking $A$ really large, and $C_2=C_4(k_1^2+\dots+k_r^2)^2$ for $C_4$ very big, we can make this as small as we want. In particular we can have $$\sum_{\substack{n >N, \\ n \in S(Y)}} \frac{|\mathbf{h}_{k_1,\dots,k_r}(n)|^2}{n^{2\sigma}} \leq \frac{B}{2}\left(\sigma-\frac12\right)^{-(k_1^2+\dots+k_r^2)},$$ which going back to (\ref{rankinLowerBoundStart}) yields
		$$\sum_{n \leq N} \frac{|\mathbf{h}_{k_1,\dots,k_r}(n)|^2}{n^{2\sigma}} \geq \frac{B}{2} \left(\sigma-\frac12\right)^{-(k_1^2+\dots+k_r^2)}$$ 
		
		To finish the proof, we observe that the criteria $\sigma \geq \frac12 + \frac{C_2}{(k_1^2+\dots+k_r^2)\log Y}$ translates to $\sigma \geq \frac12 + \frac{C_4A(k_1^2+\dots+k_r^2)}{\log N}$, so we choose $C_{k_1,\dots,k_r}=C_4A(k_1^2+\dots+k_r^2).$ We also take $\frac{C_1}{k_1^2+\dots+k_r^2}=C_{k_1,\dots,k_r}$ in the proof of the upper bound.
		
	\end{proof}
	
	\section{Lower bound for moments}
	We will need the following proposition in this section.

	\begin{prop}\label{generalizedGabriel1}
		Let $\alpha,\beta,\gamma$ be given such that $\alpha \leq \gamma \leq\beta$ with $\beta-\alpha>C$ for some fixed positive constant $C$. Assume that $f(z)$ is a holomorphic function in the strip $\alpha+\varepsilon < \textup{Re}(z) < \beta$, continuous on $\textup{Re}(z)=\beta$ and $\textup{Re}(z)=\alpha+\varepsilon$ for any $\varepsilon>0$. Suppose $f(z)\to 0$ as $|\textup{Im}(z)|\to \infty$. Assume that $$\lim_{\varepsilon \searrow 0} \int_{-\infty}^\infty |f(\alpha+\varepsilon+it)|^2 \, \textup{d}t$$ exists. Then
		$$\int_{-\infty}^{\infty} |f(\gamma+it)|^2 \, \textup{d}t \leq \left(\int_{-\infty}^\infty |f(\alpha+it)|^{2}\right)^{\frac{\beta-\gamma}{\beta-\alpha}} \left(\int_{-\infty}^\infty |f(\beta+it)|^2\right)^{\frac{\gamma-\alpha}{\beta-\alpha}}.$$
	\end{prop}
	
	\begin{proof} 
		Let $\varepsilon>0$ be small. Gabriel's convexity theorem \cite[Theorem 2]{Gabriel1} gives us that
		$$\int_{-\infty}^{\infty} |f(\gamma+\varepsilon+it)|^2 \, \textup{d}t \leq \left(\int_{-\infty}^\infty |f(\alpha+\varepsilon+it)|^{2}\right)^{\frac{(\beta-\varepsilon)-(\gamma+\varepsilon)}{(\beta-\varepsilon)-(\alpha+\varepsilon)}} \left(\int_{-\infty}^\infty |f(\beta-\varepsilon+it)|^2\right)^{\frac{(\gamma+\varepsilon)-(\alpha+\varepsilon)}{(\beta-\varepsilon)-(\alpha+\varepsilon)}}.$$ We now let $\varepsilon \to 0$, which by continuity of the integral $\int_{-\infty}^\infty |f(\sigma+it)|^2 \, \text{d}t$ in $\sigma$, implies the desired conclusion.
	\end{proof}
	
	We are now ready to prove Theorem \ref{mainThmLowerBound}. Define $$S_N(s)\coloneq\sum_{n\leq N} \frac{\mathbf{h}_{k_1,\dots,k_r}(n)}{n^s}, \qquad g_N(s)\coloneq L(s,\pi_1)^{k_1}\cdots L(s,\pi_r)^{k_r} - S_N(s).$$ and also the function $$w(t,T)\coloneq\int_T^{2T} e^{-(t-\tau)^2}\,\text{d}\tau,$$ which will concentrate our integrals to more or less only $[1,T]$.
	
	To the functions above we associate the three following integrals:
	\begin{align*}
		H(\sigma,T) &\coloneq \int_{-\infty}^{\infty} |S_N(\sigma+it)|^2w(t,T)\,\text{d}t, \\
		K(\sigma,T) &\coloneq \int_{-\infty}^{\infty} |g_N(\sigma+it)|^2w(t,T)\, \text{d}t,\\
		J(\sigma,T) &\coloneq \int_{-\infty}^\infty |L(1/2+it,\pi_1)|^{2k_1}\cdots|L(1/2+it,\pi_r)|^{2k_r}w(t,T)\,\text{d}t.
	\end{align*}
	Theorem \ref{mainThmLowerBound} will follow if we can prove $T(\log T)^{k_1^2+\dots+k_r^2} \ll J(1/2,T)$.
	
	\begin{lem}
		Assume $J(1/2,T)\gg T(\log T)^{k_1^2\dots+k_r^2}$. Then $I_{k_1,\dots,k_r}(\pi_1,\dots,\pi_r) \gg T(\log T)^{k_1^2\dots+k_r^2}$.
	\end{lem}
	
	\begin{proof}
		A direct computation shows that $$\int_0^{3T} |L(\tfrac12+it,\pi_1)|^{2k_1}\cdots|L(\tfrac12+it,\pi_r)|^{2k_r}\, \text{d}t \gg J(1/2,T),$$ because $w(t,T)\ll 1$ for all $t$, and for $t \in \mathbb{R}\backslash [0,3T]$ we estimate $w(t,T)$ explicitly and bound our $L$-functions using GLH.
	\end{proof}
	
	We now set out to prove that $J(1/2,T)\gg T(\log T)^{k_1^2+\dots+k_r^2}$. In the following section, $C$ is a positive constant that may change value from line to line. Integrating the inequality $$|S_N(s)|^2 \ll |L^{k_1}(s,\pi_1)\cdots L^{k_r}(s,\pi_r)|^2 + |g_N(s)|^2,$$ gives 
	\begin{equation}\label{HJKGeneralIneq1}
		H(\sigma,T)\ll J(\sigma,T) + K(\sigma,T).
	\end{equation}
	The same argument also gives
	\begin{equation}\label{KJHGeneralIneq2}
		K(\sigma,T)\ll J(\sigma,T) + H(\sigma,T).
	\end{equation}
	We divide the proof in two cases, depending essentially on how well $S_N(s)$ approximates $L^{k_1}(s,\pi_1)\cdots L^{k_r}(s,\pi_r)$. More specifically, we will consider the two cases $K(1/2,T)<T$ or $K(1/2,T)\geq T$. In the first case we have a good approximation, and the proof is rather simple.
	
	\textbf{Case I:} $K(1/2,T)<T$. In this case (\ref{HJKGeneralIneq1}) reads $$H(1/2,T)\ll J(1/2,T) + T,$$ so $H(1/2,T)\gg T(\log T)^{k_1^2+\dots+k_r^2}$ would imply the desired conclusion. 
	
	\begin{lem}
		Let $N\ll T$ and $\log N \gg \log T$. Assume $\frac12 + \frac{C_{k_1,\dots,k_r}}{\log N} \leq \sigma \leq \frac34$. Then $$H(\sigma,T) \gg T\left(\sigma-\frac12\right)^{-(k_1^2+\dots+k_r^2)}.$$
		Under the same assumptions, we also have $H(1/2,T)\gg T(\log T)^{k_1^2+\dots+k_r^2}$.
	\end{lem}
	
	\begin{proof}
		We start by observing that $w(t,T)\ll 1$ for all $t$, and $w(t,T) \gg 1$ for $4T/3 \leq t \leq 5T/3$. This observation and Montgomery--Vaughan mean value theorem gives $$H(\sigma,T)\gg \int_{4T/3}^{5T/3} |S_N(\sigma+it)|^2\,\textup{d}t \gg T\sum_{n\leq N} \frac{|\mathbf{h}_{k_1,\dots,k_r}(n)|^2}{n^{2\sigma}}.$$ Lemma \ref{h_kCoeffsAsymps} now gives the desired conclusion.
	\end{proof}
	
	\textbf{Case II:} $K(1/2,T)\geq T$. We start by applying Proposition \ref{generalizedGabriel1} to the function $g_N(z)e^{\frac12 (z-i\tau)^2}$ with $\alpha=1/2$, $\beta=7/8$ and $\gamma=\sigma$. This gives us 
	\begin{align}\label{lowerBoundMomentCaseIStart}
		&\int_{-\infty}^\infty |g_N(\sigma+it)e^{\frac12(\sigma+i(t-\tau))^2}|^2 \, \text{d}t \\
		\leq &\left(\int_{-\infty}^{\infty}\left|g_N(1/2+it)e^{\frac12(1/2+i(t-\tau))^2}\right|^2 \,\text{d}t\right)^{(7-8\sigma)/3}\left(\int_{-\infty}^{\infty}\left|g_N(7/8+it)e^{\frac12(7/8+i(t-\tau))^2}\right|^2 \,\text{d}t\right)^{\frac{8\sigma-4}{3}} \nonumber.
	\end{align}
	
	Since $\mathbf{h}_{k_1,\dots,k_r}(n) \ll n^{\varepsilon}$ for any $\varepsilon>0$, we have upon choosing $\varepsilon=1/2$, the following estimate: $$S_N(s)\ll N \ll T,$$ for any $\frac12 \leq \sigma \leq \frac54$. By GLH, $$|L^{k_1}(s,\pi_1)\cdots L^{k_r}(s,\pi_r)|\ll |t|^{(k_1m_1+\dots+k_rm_r)\varepsilon},$$ holds for any $\varepsilon>0$. Choosing $\varepsilon=(k_1m_1+\dots+k_rm_r)^{-1}$, we conclude that $$g_N(s)\ll T+t$$ for $|s-1|\geq \frac{1}{10}$ (to stay away from a possible pole) and $\frac12 \leq \sigma \leq \frac54$. Using these bounds we have $$\int_{-\infty}^{\tau/2} \left|g_N(7/8+it)e^{\frac12(7/8+i(t-\tau))^2}\right|^2 \,\text{d}t \ll T^2e^{-C\tau^2}$$ for some positive $C>0$ and similarly $$\int_{3\tau/2}^{\infty} \left|g_N(7/8+it)e^{\frac12(7/8+i(t-\tau))^2}\right|^2 \,\text{d}t \ll T^2e^{-C\tau^2}.$$ Thus 
	\begin{align}\label{generalizedGabriel1Conclusion}
		&\int_{-\infty}^{\infty} \left|g_N(7/8+it)e^{\frac12(7/8+i(t-\tau))^2}\right|^2 \,\text{d}t \\
		= &\int_{\tau/2}^{3\tau/2} \left|g_N(7/8+it)e^{\frac12(7/8+i(t-\tau))^2}\right|^2 \,\text{d}t + O\left(T^2e^{-C\tau^2} \right). \nonumber
	\end{align}
	
	Let us write $f(s)=g_N(s)e^{\frac12(s-i\tau)^2}$ for convenience from now on. Let furthermore $$F(z)=f\left(z+\frac78 +i\tau\right)$$ and $$z_0=\frac38 +\frac12 i\tau.$$ With this notation we have the following lemma originally due to Gabriel \cite{Gabriel1} in a more general setting.
	
	\begin{lem}\label{generalizedGabriel2}
		Let $F(z)$ be the function displayed above. Let $R$ be the rectangle whose vertices are $z_0,\overline{z_0},-z_0$ and $-\overline{z_0}$. Let $L$ be the line connecting $-i\textup{Im}(z_0)$ and $i\textup{Im}(z_0)$; let $P_1$ consist of the three line segments connecting $-i\textup{Im}(z_0), \overline{z_0},z_0$ and $i\textup{Im}(z_0)$ respectively; and let finally $P_2$ be the mirror image of $P_1$ about $L$. Then $$\int_L |F(z)|^2 \, |\textup{d}z| \leq \left(\int_{P_1} |F(z)|^2\, |\textup{d}z|\right)^{1/2} \left(\int_{P_2} |F(z)|^2\, |\textup{d}z|\right)^{1/2}.$$
	\end{lem}
	
	\begin{proof}
		Let $\varepsilon>0$ be small. Let $R_{\varepsilon}$ be $R$ but with $z_0$ interchanged with $z_0-\varepsilon$. We call the new $P_1$ and $P_2$ for $P_{1,\varepsilon}$ and $P_{2,\varepsilon}$ respectively. Since we are assuming GRH, $F$ is single-valued and an analytic function in $R_{\varepsilon}$. This also has to be the case for $\overline{F(-\overline{z})}$. Observe that on $L$, $\overline{F(z)}=\overline{F(-\overline{z})}$. Cauchy's theorem then implies $$\int_{L} |F(z)|^2 \, |\text{d}z| = \left|\int_{P_{1,\varepsilon}} F(z)\overline{F(-\overline{z})}\, \text{d}z\right| \leq \int_{P_{1,\varepsilon}} |F(z)||F(-\overline{z})|\, |\text{d}z|.$$ By Cauchy--Schwarz and the definition of $P_2$ we have 
		$$\int_{P_{1,\varepsilon}} |F(z)||F(-\overline{z})|\, |\text{d}z| \leq \left(\int_{P_{1,\varepsilon}} |F(z)|^2\, |\text{d}z| \right)^{1/2}\left(\int_{P_{2,\varepsilon}} |F(z)|^2\, |\text{d}z| \right)^{1/2}.$$ Now we let $\varepsilon \to 0$, which by continuity of the integral as a function of $\text{Re}(z)$, implies the theorem.
	\end{proof}
	
	With the notation from the proposition we have $$\int_L |F(z)|^2 \, |\text{d}z|=\int_{\tau/2}^{3\tau/2} |f(\tfrac78+it)|^2\,\text{d}t.$$ Furthermore $$\int_{P_1} |F(z)|^2 \, |\text{d}z| = \int_{\tau/2}^{3\tau/2} |f(\tfrac54 +it)|^2\, \text{d}t + \int_{\tfrac78}^{\tfrac54} \left(|f(\eta+\tfrac12 i \tau)|^2 + |f(\eta+\tfrac32 i\tau)|^2\right)\, \text{d}\eta.$$ Using the same bound as earlier, $g_N(s)\ll T+t$, we bound the last integral by $$\ll T^2e^{-C\tau^2}.$$ Thus $$\int_{P_1} |F(z)|^2\, |\text{d}z| = \int_{\tau/2}^{3\tau/2} |f(\tfrac54+it)|^2\,\text{d}t + O(T^2e^{-C\tau^2}).$$ By the exact same reasoning, $$\int_{P_2} |F(z)|^2 \, |\text{d}z| = \int_{\tau/2}^{3\tau/2} |f(\tfrac12 +it)|^2 \, \text{d}t + O(T^2e^{-c\tau^2}).$$ Thus our application of Proposition \ref{generalizedGabriel2} reads 
	\begin{align}\label{generalizedGabriel2Conclusion}
		&\int_{\tau/2}^{3\tau/2} |f(\tfrac78+it)|^2\, \text{d}t \\ 
		\ll &\left(\int_{\tau/2}^{3\tau/2} |f(\tfrac12+it)|^2\, \text{d}t\right)^{1/2}\left(\int_{\tau/2}^{3\tau/2} |f(\tfrac54+it)|^2\, \text{d}t\right)^{1/2} + O(T^2e^{-C\tau^2}). \nonumber
	\end{align}
	Putting (\ref{generalizedGabriel1Conclusion}) together with (\ref{generalizedGabriel2Conclusion}), we get
	\begin{align}\label{momentLowerBound7/8CompleteIntegral}
		&\int_{-\infty}^\infty |f(\tfrac78+it)|^2 \, \text{d}t \\
		\ll &\left(\int_{\tau/2}^{3\tau/2} |f(\tfrac12+it)|^2\, \text{d}t\right)^{1/2}\left(\int_{\tau/2}^{3\tau/2} |f(\tfrac54+it)|^2\, \text{d}t\right)^{1/2} + O(T^2e^{-C\tau^2})\nonumber.
	\end{align}
	Going back to (\ref{lowerBoundMomentCaseIStart}), and putting (\ref{momentLowerBound7/8CompleteIntegral}) into this equation we have (as long as $1/2\leq \sigma\leq 7/8$) that
	\begin{align}\label{momentLowerBoundRightBeforeKIneq}
		&\int_{-\infty}^\infty |f(\sigma+it)|^2 \, \text{d}t \\
		\ll &\left(\int_{-\infty}^\infty |f(\tfrac12 +it)|^2 \, \text{d}t\right)^{\frac{7-8\sigma}{3}}\left(\left(\int_{\tau/2}^{3\tau/2} |f(\tfrac12+it)|^2\, \text{d}t\right)^{1/2}\left(\int_{\tau/2}^{3\tau/2} |f(\tfrac54+it)|^2\, \text{d}t\right)^{1/2}  \right)^{\frac{8\sigma-4}{3}}\nonumber \\
		&+ \left(T^2e^{-C\tau^2}\right)^{\frac{8\sigma-4}{3}}\left(\int_{-\infty}^\infty |f(\tfrac12 +it)|^2 \, \text{d}t\right)^{\frac{7-8\sigma}{3}}  \nonumber\\
		\ll &\left(\int_{-\infty}^\infty |f(\tfrac12 +it)|^2 \, \text{d}t\right)^{\frac{5-4\sigma}{3}} \left(\int_{\tau/2}^{3\tau/2} |f(\tfrac54+it)|^2\, \text{d}t\right)^{\frac{4\sigma-2}{3}} \nonumber\\
		&+ \left(T^2e^{-C\tau^2}\right)^{\frac{8\sigma-4}{3}}\left(\int_{-\infty}^\infty |f(\tfrac12 +it)|^2 \, \text{d}t\right)^{\frac{7-8\sigma}{3}}.\nonumber
	\end{align}
	
	By Hölder's inequality in the form $\int |f|^{\alpha}|g|^{\beta} \leq (\int |f|)^{\alpha} (\int |g|)^{\beta}$ for $\alpha+\beta=1$, we have, since $\frac{5-4\sigma}{3}+\frac{4\sigma-2}{3}=1$, that 
	\begin{align*}
		&\int_T^{2T} \left(\int_{-\infty}^\infty |f(\tfrac12 +it)|^2 \, \text{d}t\right)^{\frac{5-4\sigma}{3}} \left(\int_{\tau/2}^{3\tau/2} |f(\tfrac54+it)|^2\, \text{d}t\right)^{\frac{4\sigma-2}{3}} \, \text{d}\tau \\
		\leq &\left(\int_T^{2T} \int_{-\infty}^\infty |f(\tfrac12 +it)|^2 \, \text{d}t \text{d}\tau \right)^{\frac{5-4\sigma}{3}}\left(\int_T^{2T}\int_{\tau/2}^{3\tau/2} |f(\tfrac54+it)|^2\, \text{d}t\text{d}\tau\right)^{\frac{4\sigma-2}{3}} \\
		= &\left(\int_T^{2T} \int_{-\infty}^\infty |g_N(\tfrac12 +it)|^2e^{-(t-\tau)^2} \, \text{d}t \text{d}\tau \right)^{\frac{5-4\sigma}{3}}\left(\int_T^{2T}\int_{\tau/2}^{3\tau/2} |g_N(\tfrac54+it)|^2e^{-(t-\tau)^2}\, \text{d}t\text{d}\tau\right)^{\frac{4\sigma-2}{3}} \\
		= &K(1/2,T)^{\frac{5-4\sigma}{3}}\left(\int_T^{2T}\int_{\tau/2}^{3\tau/2} |g_N(\tfrac54+it)|^2e^{-(t-\tau)^2}\, \text{d}t\text{d}\tau\right)^{\frac{4\sigma-2}{3}},
	\end{align*}
	upon changing the order of integration in the end. Similary, as $\frac{7-8\sigma}{3}+\frac{8\sigma-4}{3}=1$, Hölder again gives
	\begin{align*}
		&\int_{T}^{2T}\left(T^2e^{-C\tau^2}\right)^{\frac{8\sigma-4}{3}}\left(\int_{-\infty}^\infty |f(\tfrac12 +it)|^2 \, \text{d}t\right)^{\frac{7-8\sigma}{3}} \, \text{d}\tau \\
		\leq&\left(T^2e^{-CT^2}\right)^{\frac{8\sigma-4}{3}}K(1/2,T)^{\frac{7-8\sigma}{3}}.
	\end{align*}
	
	Inserting these two applications of Hölder's inequality in (\ref{momentLowerBoundRightBeforeKIneq}) gives 
	\begin{align}\label{momentLowerAlmostKIneq}
		K(\sigma,T) &\ll K(1/2,T)^{\frac{5-4\sigma}{3}}\left(\int_T^{2T}\int_{\tau/2}^{3\tau/2} |g_N(\tfrac54+it)|^2e^{-(t-\tau)^2}\, \text{d}t\text{d}\tau\right)^{\frac{4\sigma-2}{3}} \\ 
		&+ \left(T^2e^{-CT^2}\right)^{\frac{8\sigma-4}{3}}K(1/2,T)^{\frac{7-8\sigma}{3}}. \nonumber
	\end{align}
	To estimate the last factor in the first term in the displayed equation above, we use Montgomery--Vaughan and the Dirichlet series definition of $L^{k_1}(s,\pi_1)\cdots L^{k_r}(s,\pi_r)$ which is valid $\text{Re}(s)>1$: 
	\begin{align*}
		\int_T^{2T}\int_{\tau/2}^{3\tau/2} |g_N(\tfrac54+it)|^2e^{-(t-\tau)^2}\, \text{d}t\text{d}\tau &\ll \int_{T/2}^{3T/2} |g_N(\tfrac54+it)|^2\,\text{d}t \\
		&\ll T\sum_{n>N} \frac{|\mathbf{h}_{k_1,\dots,k_r}(n)|^2}{n^{5/2}} \ll TN^{\varepsilon-3/2},
	\end{align*}
	where we again have used in the end that $\mathbf{h}_{k_1,\dots,k_r}(n) \ll n^{\varepsilon}$. In light of this, (\ref{momentLowerAlmostKIneq}) reads 
	\begin{align*}
		K(\sigma,T) &\ll K(1/2,T)^{\frac{5-4\sigma}{3}}T^{\frac{4\sigma-2}{3}}N^{(\varepsilon-\frac32)\frac{4\sigma-2}{3}} + \left(T^2e^{-CT^2}\right)^{\frac{8\sigma-4}{3}}K(1/2,T)^{\frac{7-8\sigma}{3}} \\
		&\ll K(1/2,T)^{\frac{5-4\sigma}{3}}T^{\frac{4\sigma-2}{3}}N^{(\varepsilon-\frac32)\frac{4\sigma-2}{3}} + K(1/2,T)^{\frac{7-8\sigma}{3}}
	\end{align*}
	Recall that we assumed $K(1/2,T)\geq T$ for this case, so with the choices $\varepsilon=1/2$, $N=T^{1/2}$, the inequality reads as: 
	\begin{equation}\label{KFinalIneq}
		K(\sigma,T) \ll K(1/2,T)T^{\frac{1-2\sigma}{3}},
	\end{equation}
	for $\tfrac12 \leq \sigma \leq \tfrac34$.
	
	With all the information we need about $K(\sigma,T)$ proved, we now tidy things up and finish the proof. By (\ref{HJKGeneralIneq1}), (\ref{KJHGeneralIneq2}) and (\ref{KFinalIneq}) we have $$H(\sigma,T) \ll J(\sigma,T)+(H(1/2,T)+J(1/2,T))T^{\frac{1-2\sigma}{3}}.$$
	Thus at least one of the following inequalities have to hold:
	$$H(\sigma,T) \ll H(1/2,T)T^{\frac{1-2\sigma}{3}} \qquad \text{or} \qquad H(\sigma,T) \ll J(\sigma,T)+J(1/2,T)T^{\frac{1-2\sigma}{3}}.$$
	Assume the first inequality holds. We take $\sigma=\tfrac12 + \frac{\eta}{\log T}$ for some $\eta>C_{k_1,\dots,k_r}$. By Lemma \ref{h_kCoeffsAsymps}, the displayed inequality above reads $$e^{\eta}\leq D_{k_1,\dots,k_r} \eta^{k_1^2+\dots+k_r^2}$$ for some fixed constant $D_{k_1,\dots,k_r}>0$. Taking $\eta$ sufficiently big, we contradict this inequality. Thus we conclude that for this $\eta$, 
	\begin{equation}\label{momentLowerBoundAlmostFinished}
		H(\widetilde{\sigma},T)\leq J(\widetilde{\sigma},T)+J(1/2,T)T^{\frac{1-2\sigma}{3}}
	\end{equation}
	must hold when $\widetilde{\sigma}=\frac{1}{2}+\frac{\eta}{\log T}$. We now claim that $$J(\widetilde{\sigma},T)\ll T^{\widetilde{\sigma}-\frac12}J(1/2,T)^{\frac32 -\widetilde{\sigma}} $$ for some positive constant $C>0$. Assuming the truth of this, we get from (\ref{momentLowerBoundAlmostFinished}) and Lemma \ref{h_kCoeffsAsymps} that $$T(\log T)^{k_1^2\dots+k_r^2} \ll J(1/2,T),$$ which is just what we wanted to prove. Thus the only thing left to prove to finish the proof is our claim above. It will follow from
	
	\begin{lem}
		Let $\tfrac12 \leq \sigma \leq \tfrac34$. Then  $J(\sigma,T)\ll T^{\sigma-1/2}J(1/2,T)^{3/2-\sigma}$.
	\end{lem}
	
	\begin{proof}
		There are two cases to consider. If none of the $L$-functions $L(s,\pi_1),\dots,L(s,\pi_r)$ are the Riemann zeta function, then we have no pole at $s=1$ for the function $f(s)=L^{k_1}(s,\pi_1)\cdots L^{k_r}(s,\pi_r)e^{(s-i\tau)^2/2}$. In the other case, where $L(s,\pi_1)=\zeta(s)$ say, we consider the function $f(s)=(s-1)^{k_1}L^{k_1}(s,\pi_1)\cdots L^{k_r}(s,\pi_r)e^{(s-i\tau)^2}$. In any case, we apply Proposition \ref{generalizedGabriel1} to $f$ with $\alpha=1/2$, $\beta=3/2$ and $\gamma=\sigma$. The procedure for both the cases are very similar, the second being a tiny bit more computational. For the first case, the result is immediate after applying Proposition \ref{generalizedGabriel1} and Jensen's inequality. For the second case, we follow the same steps as Heath-Brown \cite[p. 71]{Heath-Brown1} and also apply Jensen's inequality in the end.
	\end{proof}

	\section{Upper bound for moments}
	
	{In this section we work with the moments $$\int_{T}^{2T} \prod_{j=1}^r|L(\tfrac12+it,\pi_j)|^{2k_j} \dt.$$ Theorem \ref{mainThmUpperBound} then follows from a dyadic decomposition.} Our point of departure will be the following lemma due to Chandee \cite[Theorem 2.1]{Chandee1}.
	\begin{lem}\label{chandeeLem}
		Assume that $L(s,\pi)$ is the Riemann zeta function or that the completed $L$-function $\Phi(s,\pi)$ of $L(s,\pi)$ has no pole or zero at $s=0,1$. Assume that RH and the Generalized Ramanujan conjecture holds for $L(s,\pi)$. Assume furthermore that $L$ has degree $m$. Then for any $e^2 \leq x \leq T^2$ and any $T \leq t \leq 2T$, for sufficiently large $T$, we have $$\log \left|L \left(\frac12 +it,\pi\right) \right| \leq \textup{Re} \left(\sum_{\substack{n\leq x,\\n=p,p^2}}\frac{\Lambda_{\pi}(n)}{n^{1/2+1/\log x+it}\log n}\frac{\log (x/n)}{\log x} \right)+m\frac{\log T}{\log x}+O(1).$$
	\end{lem}
	
	\begin{proof}
		We will use the formulation from \cite[Lemma 3.3]{MilinoButter1}. Choosing $\lambda=1$ therein gives that $$\log \left|L\left(1/2+it,\pi\right) \right| \leq \text{Re}\left(\sum_{n \leq x} \frac{\Lambda_{\pi}(n)}{n^{1/2+1/\log x+it}\log n}\frac{\log (x/n)}{\log x}\right)+ \frac{m\log T}{\log x}+O\left(\frac{1}{\log x}\right)$$ for $t \in [T,2T]$. By means of (\ref{conditionalVonMangoldtBound}), we see that $$\sum_{p^j \leq x, j\geq 3}\left|\frac{\Lambda_{\pi}(p^j)}{(p^j)^{1/2+1/\log x+it}\log p^j}\frac{\log (x/p^j)}{\log x}\right| \ll \sum_{p^j \leq x, j\geq 3}\frac{1}{jmp^{j/2}} = O(1).$$
	\end{proof}

	We keep the same representations $\pi_1,\dots,\pi_r$ as in the previous section. Recall that they were of degree $m_1,\dots,m_r$ respectively. Following Harper we shall split $[T,2T]$ into many sets depending on how small (or large) the magnitude of certain pieces of the Dirichlet polynomial $$\sum_{p \leq x} \frac{k_1\Lambda_{\pi_1}(p)+\dots+k_r\Lambda_{\pi_r}(p)}{p^{1/2+1/\log x+it}\log(p)}\frac{\log(x/p)}{\log x}$$ is when evaluated in $t$. It will be very convenient to have a shorthand notation for the coefficients appearing in the Dirichlet polynomial above. Therefore we define $$\mathbf{\Lambda}_x(n)\coloneq\frac{k_1\Lambda_{\pi_1}(n)+\dots+k_r\Lambda_{\pi_r}(n)}{\log(n)}\frac{\log(x/n)}{n^{1/\log x}\log x}.$$ Observe that $\log(x/n)/(n^{1/\log x}\log x)\leq 1$ for $n\leq x$. We are now ready to set up all the necessary notation for the proof.
	\begin{align*}
		\widehat{k} &\coloneq  \max \left\{\sum_{i=1}^r m_ik_i, \sum_{i=1}^r k_i^2\right \} \\
		\theta_j &\coloneq \frac{e^{j-1}}{(\log \log T)^2} \\
		T_j &\coloneq \left\{\begin{matrix}
			1 & j=0\\ 
			T^{\theta_j} & j\geq 1
		\end{matrix}\right. \\
		K_j &\coloneq \widehat{k}^{1/2}\theta_j^{-3/4}.
	\end{align*}
	Define $J$ to be the largest integer such that $$T_J \leq T^{e^{-1000\widehat{k}}}.$$ Observe in particular that this implies $J-1 \leq 2\log_3T - 1000\widehat{k}\asymp \log_3 T$. Let 
	\begin{align*}
		\mathcal{P}_{j,x}(s)&\coloneq\sum_{T_{j-1}<p\leq T_j} \frac{\mathbf{\Lambda}_x(p)}{p^s}, \\
		\mathcal{N}_{j,x}(s) &\coloneq \sum_{\substack{\Omega(n)\leq 10K_j, \\ p \mid n \implies T_{j-1}<p\leq T_j}} \frac{\mathfrak{g}_x(n)}{n^s},
	\end{align*}
	where $\mathfrak{g}_x(n)$ is defined on prime powers by $$\mathfrak{g}_x(p^\eta)=\frac{1}{\eta!}\mathbf{\Lambda}_x(p)^\eta,$$ and extended multiplicatively. A simple computation shows that the length of the full product $\prod_{j=1}^{\mathscr{J}} |\mathcal{N}_{j,x}(\tfrac12+it)|^2$ is $\leq T^{6/10}$.
	
	If $\mathcal{P}_{j,x}(s)$ is of ``typical'' size, we can estimate $\exp(\mathcal{P}_{j,x}(s))$ rather sharply by $\mathcal{N}_{j,x}(s)$. We capture this behaviour by the following set: $$\mathscr{G} \coloneq \left\{ t \in [T,2T] : \left|\mathcal{P}_{j,T_J}\left(\frac12 + it\right) \right| \leq K_j \enspace \forall 1 \leq j \leq J \right\}.$$ On some other sets we will have typical behaviour of the $\mathcal{P}_{j,x}$ for all $j$ up to some index $\ell$, then some irregular behaviour after that. We capture this behaviour by the following sets
	$$\mathscr{B}_j \coloneq \left\{t \in [T,2T]:\begin{split}
		\left|\mathcal{P}_{r,T_s}\left(\frac12 +it\right)\right|&\leq K_r \enspace \forall 1\leq r < j \text{ and } r\leq s \leq J, \text{but} \\ \left|\mathcal{P}_{j,T_s}\left(\frac12 +it\right) \right| &> K_j \text{ for some } j \leq s \leq J.
	\end{split} \right\}.$$
	Observe that $$[T,2T]=\mathscr{G} \cup \bigcup_{j=1}^J \mathscr{B}_j.$$ It will also be nice to further decompose $\mathscr{B}_j$ into sets 
	$$\mathscr{B}_{j,\ell} \coloneq \left\{t \in [T,2T] : \begin{split}
		|\mathcal{P}_{r,T_{j-1}}(\tfrac12+it)| &\leq K_r \enspace \forall 1\leq r < j,  \\ \textup{but } |\mathcal{P}_{j,T_{\ell}}(\tfrac12 +it)|&>K_j 
	\end{split}\right\}.$$  This concludes the setup of all the necessary notation. 
	
	We shall need the following basic lemma.
	
	\begin{lem}\label{RadziwillTruncationLemma}
		Assume we are given a Dirichlet polynomial $D(s)=\sum_{p \leq Y} a(p)p^{-s}$. If $t \in [T,2T]$ is such that $|D(s)|\leq V$, then $$\exp\left(2\textup{Re} D(s)\right) = (1+O(e^{-9V}))\left|\sum_{j \leq 10V} \frac{(D(s))^j}{j!}\right|^2.$$
		In particular if $|\mathcal{P}_{j,X}(\tfrac12+it)|\leq K_j$ for $1\leq \mathscr{J} \leq J$, then  $$\prod_{j=1}^{\mathscr{J}}\exp\left(2\textup{Re}\mathcal{P}_{j,X}(\tfrac12 +it)\right) \ll \prod_{j=1}^{\mathscr{J}} |\mathcal{N}_{j,x}(\tfrac12+it)|^2$$ 
	\end{lem}
	
	\begin{proof}
		The first part follows from \cite[Eq. 37]{Heap1}. Applying the first part of the lemma, to each of the $\exp\left(2\textup{Re}\mathcal{P}_{j,X}(\tfrac12 +it)\right)$, then using the multinomial theorem, we obtain the desired result.
	\end{proof}
	
	\begin{prop}\label{upperBoundInitialIneq}
		For any positive integers $s_j$ we have
		\begin{align*}
			&\int_{T}^{2T} \prod_{i=1}^r |L(\tfrac12+it,\pi_i)|^{2k_i} \dt \\
			\ll &\int_{T}^{2T} |\mathcal{M}_{T_J}(1+2it)|^2\prod_{j=1}^J  |\mathcal{N}_{j,T_J}(\tfrac12+it)|^2 \dt \\
			+ &\sum_{\substack{1 \leq j \leq J, \\ j\leq \ell \leq J}}\int_{T}^{2T}|\mathcal{M}_{T_\ell}(1+2it)|^2\exp\left(\frac{2\widehat{k}}{\theta_{j-1}}\right)\left(\frac{|\mathcal{P}_{j,T_\ell}(\tfrac12+it)|}{K_j}\right)^{2s_j}\prod_{i=1}^{j-1} |\mathcal{N}_{i,T_\ell}(\tfrac12+it)|^2\dt  \\
			+ \,\, &O((\log \log \log T)^2T(\log T)^{-1})
		\end{align*}
		where $$\mathcal{M}_x(s)=\sum_{\substack{\Omega(n)\leq 10(\log\log T)^2, \\ p \mid n \implies p \leq \log T}} \frac{\mathfrak{h}_x(n)}{n^{s}},$$ and $\mathfrak{h}_x(n)$ is multiplicative, defined on prime powers by $$\mathfrak{h}_x(p^\eta) =\frac{1}{\eta!}\mathbf{\Lambda}_x(p^2)^\eta.$$
	\end{prop}
	
	\begin{proof}
		Starting with some $t \in [T,2T]$ we have the following inequality 
		\begin{equation}\label{prodOfLFuncsIneq}
			\prod_{i=1}^{r} |L(\tfrac12+it,\pi_i)|^{2k_i} \leq \mathbf{1}_{t \in \mathscr{G}}\cdot \prod_{i=1}^{r} |L(\tfrac12+it,\pi_i)|^{2k_i} + \sum_{1 \leq j \leq J} \mathbf{1}_{t \in \mathscr{B}_j} \cdot \prod_{i=1}^{r} |L(\tfrac12+it,\pi_i)|^{2k_i}.
		\end{equation}
		We now aim to apply Lemma \ref{chandeeLem} to all of the terms above by themselves. We shall start by showing that we can take the sum over prime squares in this approximation to be very short. To this end, define $$E_x\coloneq \left \{t \in [T,2T] : \left|\sum_{(\log T)^{5(k_1^2+\dots+k_r^2+1)}<p\leq x} \frac{\mathbf{\Lambda}_x(p^2)}{p^{1+2it}}\right| > k_1m_1+\dots+k_rm_r \right\},$$
		where we assume $x>(\log T)^{5(k_1^2+\dots+k_r^2+1)}$ (which will be the case in our applications). Applying Chebyshev's inequality and GRC, one sees that $\text{meas}(E_x)\ll T(\log T)^{-5(k_1^2+\dots+k_r^2+1)}$. By Cauchy--Schwarz and bounds for the $4(k_1,\dots,k_r)$th moment given in \cite[Theorem 1.1.]{MilinoButter1}, we thus get that $$\int_{E_x} \prod_{i=1}^r |L(\tfrac12+it,\pi_1)|^{2k_i}\dt \ll T(\log T)^{-1}.$$ Applying Lemma \ref{chandeeLem} with $x=T_J$, alongside the bound above we get that
		\begin{align*}
			&\int_{\mathscr{G}} \prod_{i=1}^{r} |L(\tfrac12+it,\pi_i)|^{2k_i} \dt \\
			\leq &\int_{\mathscr{G} \cap E_{T_J}^c} \prod_{i=1}^r|L(\tfrac12+it,\pi_i)|^{2k_i} \dt + \int_{E_{T_J}} \prod_{i=1}^r |L(\tfrac12+it,\pi_i)|^{2k_i} \dt \\
			\ll &\int_{\mathscr{G} \cap E_{T_J}^c} \exp \left(2\text{Re}\sum_{\substack{n \leq T_J,\\n=p,p^2}} \frac{\mathbf{\Lambda}_{T_J}(n)}{n^{1/2+it}} + \frac{\widehat{k}\log T}{\log T_J} \right) \dt + O(T(\log T)^{-1}). 
		\end{align*}
		GRC and Mertens theorem implies that $$\left|\sum_{\log T <p\leq(\log T)^{5(k_1^2+\dots+k_r^2+1)}} \frac{k_1\Lambda_{\pi_1}(p^2)+\dots+k_r\Lambda_{\pi_r}(p^2)}{p^{1+2it+2/\log T_J}}\frac{\log(T_J/p^2)}{\log T_J}\right|=O(1).$$ Combining this with the observation that $\frac{\widehat{k}\log T}{\log T_J}=\frac{\widehat{k}}{\theta_J} = O(1)$, we see that
		\begin{align*}
			&\int_{\mathscr{G}\cap E_{T_J}^c} \exp \left(2\text{Re}\sum_{\substack{n \leq T_J,\\n=p,p^2}} \frac{\mathbf{\Lambda}_{T_J}(p)}{p^{1/2+it}}+\frac{\widehat{k}\log T}{\log T_J}\right)\dt \\
			\ll &\int_{\mathscr{G} \cap E_{T_J}^c} \exp\left(2\text{Re}\left(\sum_{\substack{p \leq T_J}} \frac{\mathbf{\Lambda}_{T_J}(p)}{p^{1/2+it}} + \sum_{p \leq \log T} \frac{\mathbf{\Lambda}_{T_J}(p^2)}{p^{1+2it}}\right)\right)\dt.
		\end{align*}
		By Lemma \ref{RadziwillTruncationLemma} the integrand can be upper bounded in the following way:
		\begin{align*}
			\exp\left(2\text{Re}\left(\sum_{\substack{p \leq T_J}} \frac{\mathbf{\Lambda}_{T_J}(p)}{p^{1/2+it}}\right)\right) &\ll \prod_{j=1}^J |\mathcal{N}_{j,T_J}(\tfrac12+it)|^2, \\
			\exp\left(2\text{Re}\left(\sum_{p \leq \log T} \frac{\mathbf{\Lambda}_x(p^2)}{p^{1+2it}}\right)\right) &\ll |\mathcal{M}_x(t)|^2.
		\end{align*}
		Piecing the foregoing page of computations together, we conclude with the following inequality
		\begin{equation}\label{goodSetIneqSetupLemma}
			\int_{\mathscr{G}} \prod_{i=1}^{r} |L(\tfrac12+it,\pi_i)|^{2k_i} \dt \ll \int_{T}^{2T} |\mathcal{M}_{T_J}(1+2it)|^2\prod_{j=1}^J  |\mathcal{N}_{j,T_J}(\tfrac12+it)|^2 \dt + O(T(\log T)^{-1}),
		\end{equation}
		where we extended the integration range in the end by positivity of the integrand.
		
		We follow the procedure above when integrating over $\mathscr{B}_j$, except some minor changes. First and foremost, in Lemma \ref{chandeeLem} we instead choose $x=T_{j-1}$. We also split up the $\mathscr{B}_j$ into the sets $\mathscr{B}_{j,\ell}$: $$\sum_{1 \leq j \leq J} \mathbf{1}_{t \in \mathscr{B}_j} \cdot \prod_{i=1}^{r} |L(\tfrac12+it,\pi_i)|^{2k_i} \leq \sum_{\substack{1 \leq j \leq J, \\ j\leq \ell \leq J}} \mathbf{1}_{t \in \mathscr{B}_{j,\ell}}\prod_{i=1}^r|L(\tfrac12+it,\pi_i)|^{2k_i}.$$
		On $\mathscr{B}_{j,\ell}$, $$\left(\frac{\mathcal{P}_{j,T_\ell} \left(\tfrac12+it\right)}{K_j}\right)^{2s_j} >1$$ for any integer $s_j$. On each $\mathscr{B}_{j,\ell}$ we can only apply Lemma \ref{RadziwillTruncationLemma} up to the index $j-1$. The final inequality is 
		\begin{align}
			&\int_{\mathscr{B}_{j,\ell}} \prod_{i=1}^r |L(\tfrac12+it,\pi_i)|^{2k_i} \dt\\
			\ll  &\int_{T}^{2T}|\mathcal{M}_{T_\ell}(1+2it)|^2\exp\left(\frac{2\widehat{k}}{\theta_{j-1}}\right)\left(\frac{|\mathcal{P}_{j,T_\ell}(\tfrac12+it)|}{K_j}\right)^{2s_j}\prod_{i=1}^{j-1} |\mathcal{N}_{i,T_\ell}(\tfrac12+it)|^2\dt \nonumber \\
			+ &O(T(\log T)^{-1}). \nonumber
		\end{align}
		Now we sum this inequality over $1\leq j \leq J$ and $j \leq \ell \leq J$ and conclude.
	\end{proof}
	
	To finish the proof of the upper bound for the moments, we only have to compute the mean values appearing in Proposition \ref{upperBoundInitialIneq}. A tool that will greatly aid us in this endeavour is the following lemma on coprime Dirichlet polynomials.
	
	\begin{lem}\label{coprimeDirichletPolysLemma}
		Let $D_1,\cdots,D_M$ be a collection of Dirichlet polynomials, defined by $$D_j(s)\coloneq \sum_{n \in \mathcal{D}_j} \frac{\mathfrak{d}_j(n)}{n^s}.$$ Suppose that if $n \in \mathcal{D}_i$, $m \in \mathcal{D}_j$ with $i\neq j$, then $\gcd(n,m)=1$. Suppose $D_1\cdots D_M(s)$ have length $N$, then $$\frac{1}{T}\int_{T}^{2T} \left|\prod_{j=1}^{M} D_j(s)\right|^2 \dt = (1+O(NT^{-1}\log N )) \prod_{j=1}^M \left(\frac{1}{T} \int_{T}^{2T} |D_j(s)|^2 \dt \right).$$
	\end{lem}
	
	\begin{proof}
		This is \cite[Eq. 16]{HeapSound1}.
	\end{proof}
	
	Let us start with computing the mean value in the first term of Proposition \ref{upperBoundInitialIneq}.
	
	\begin{prop}\label{goodSetLemma}
		$$\int_{T}^{2T} |\mathcal{M}_{T_J}(1+2it)|^2\prod_{j=1}^J  |\mathcal{N}_{j,T_J}(\tfrac12+it)|^2 \dt\ll T(\log T)^{k_1^2+\dots+k_r^2}.$$
	\end{prop}
	
	\begin{proof}
		Let $$\mathcal{A}(t)=\mathcal{M}_{T_J}(1+2it)\mathcal{N}_{1,T_J}(1/2+it).$$ Applying\footnote{Observe that the Dirichlet polynomial in question is short enough. As mentioned earlier the big product has length $\leq T^{6/10}$. The Dirichlet polynomial $\mathcal{M}_{T_J}$ is very short, say $\leq T^{1/1000}$.} Lemma \ref{coprimeDirichletPolysLemma} we have 
		\begin{align*}
			&\frac{1}{T}\int_{T}^{2T} |\mathcal{M}_{T_J}(1+2it)|^2\prod_{j=1}^J  |\mathcal{N}_{j,T_J}(\tfrac12+it)|^2 \dt\\
			\ll &\frac{1}{T}\int_{T}^{2T}|\mathcal{A}(t)|^2\dt\cdot\prod_{j=2}^J\left(\frac{1}{T} \int_{T}^{2T} |N_{j,T_J}(\tfrac12+it)|^2 \dt\right).
		\end{align*}
		Each factor in the big product can be computed by Montgomery--Vaughan mean value theorem and simply removing the $\Omega$-condition (which we can since every term is non-negative)
		\begin{align*}
			\frac{1}{T} \int_{T}^{2T} |N_{j,T_J}(\tfrac12+it)|^2 \dt &\ll \sum_{\substack{p \mid n \implies T_{j-1}<p\leq T_j}} \frac{|\mathfrak{g}_{T_J}(n)|^2}{n} \\
			&= \exp\left(\sum_{T_{j-1}<p\leq T_j} \frac{|\mathfrak{g}_{T_J}(p)|^2}{p}+\frac{|\mathfrak{g}_{T_J}(p^2)|^2}{p^2} + \dots \right). 
		\end{align*}
		Hence it follows from the definition of $\mathfrak{g}_x(n)$, GRC and Selberg's orthogonality conjecture that  
		\begin{align*}
			\prod_{j=2}^J\left(\frac{1}{T} \int_{T}^{2T} |N_{j,T_J}(\tfrac12+it)|^2 \dt \right) &\ll \exp\left(\sum_{T_{1}<p\leq T_J} \frac{|k_1a_{\pi_1}(p)+\dots+k_ra_{\pi_r}(p)|^2}{p}\right) \\
			&\ll \left(\frac{\log T}{\log T_1}\right)^{k_1^2+\dots+k_r^2}. 
		\end{align*}

		To compute $\int_T^{2T}|\mathcal{A}(t)|^2\dt$, we realize the coefficients of $\mathcal{A}(t)$ as a convolution. For this purpose, let $\mathfrak{m}(n^2)\coloneq\mathfrak{h}_{T_J}(n)$. Then $$\mathcal{A}(t)=\sum_{n}\frac{1}{n^{1/2+it}}\sum_{\substack{m\ell=n \\ \Omega(m)\leq 10(\log \log T)^2, \\ \Omega(\ell) \leq 10K_1, \\ p \mid m \implies p \leq \log T, \\ p \mid \ell \implies p \leq T_1}} \mathbf{1}_{m \text{ is a square}}\cdot \mathfrak{m}(m)\mathfrak{g}_{T_J}(\ell).$$		
		For the sake of simpler notation, in the computation below, $\sum^{'}$ means that the sum is taken with respect to the additional conditions $p \mid m \implies p \leq \log T$ and  $p \mid \ell \implies p \leq T_1$. By non-negativity we have
		\begin{align*}
			\frac{1}{T}\int_{T}^{2T} |\mathcal{A}(t)|^2 \dt &\ll \sum_{n} \frac{1}{n}\bigg{|}\sideset{}{'}\sum_{\substack{m\ell=n \\ \Omega(m)\leq 10(\log \log T)^2, \\ \Omega(\ell) \leq 10K_1}}  \mathbf{1}_{m \text{ is a square}}\cdot \mathfrak{m}(m)\mathfrak{g}_{T_J}(\ell)\bigg{|}^2 \\
			&\leq \sum_{n} \frac{1}{n}\bigg{(}\sideset{}{'}\sum_{\substack{m\ell=n \\ \Omega(m)\leq 10(\log \log T)^2, \\ \Omega(\ell) \leq 10K_1}} \left| \mathbf{1}_{m \text{ is a square}}\cdot \mathfrak{m}(m)\mathfrak{g}_{T_J}(\ell)\right|\bigg{)}^2 \\
			&\leq \sum_{n} \frac{1}{n} \left(\sideset{}{'}\sum_{m\ell=n} \left| \mathbf{1}_{m \text{ is a square}}\cdot \mathfrak{m}(m)\mathfrak{g}_{T_J}(\ell)\right|\right)^2 \leq \sum_{p \mid n \implies p \leq T_1} \frac{\mathfrak{a}(n)^2}{n}.
		\end{align*}
		Here $\mathfrak{a}(n)$ is the multiplicative function defined on prime powers by $$\mathfrak{a}(p^\eta) \coloneq \sum_{\substack{2n_1+n_2=\eta, \\ n_1,n_2\geq 0}} |\mathfrak{m}(p^{2n_1})\mathfrak{g}_{T_J}(p^{n_2})|.$$ Observe that we have $\mathfrak{a}(p)=|\mathfrak{g}_{T_J}(p)|$. For $\eta \geq 2$, we have $$\mathfrak{a}(p^\eta) \ll \sum_{n_1+n_2=\eta} |\mathfrak{m}(p^{2n_1})\mathfrak{g}_{T_J}(p^{n_2})| \ll \left(\frac{6(k_1m_1+\dots+k_rm_r)}{\eta}\right)^\eta$$ by the multinomial theorem and GRC. Thus $$\sum_{p \mid n \implies p \leq T_1} \frac{\mathfrak{a}(n)^2}{n} \ll \prod_{p \leq T_1} \left(1+\frac{|\mathfrak{g}_{T_J}(p)|^2}{p} + O\left(\frac{1}{p^2}\right)\right) \ll (\log T_1)^{k_1^2+\dots+k_r^2}.$$
	\end{proof}
	
	For the next mean value computation we will use the following lemma on high moments of short Dirichlet polynomials.
	
	\begin{lem}\label{soundHighMomentsLemma}
		Assume $N$ is such that $N^\ell \leq T$ for some integer $\ell$. Then for coefficients $a(p) \in \mathbb{C}$, we have $$\int_{T}^{2T} \left|\sum_{p \leq N} \frac{a(p)}{p^{1/2+it}}\right|^{2\ell}\dt \ll T\ell! \left(\sum_{p \leq N} \frac{|a(p)|^2}{p}\right)^{\ell}.$$
	\end{lem}
	
	\begin{proof}
		See \cite[Lemma 3]{Sound1}.
	\end{proof}
	
	\begin{prop}
		With $s_j=\lfloor\frac{1}{10\theta_j}\rfloor$, we have
		\begin{align*}
			&\sum_{\substack{1 \leq j \leq J, \\ j\leq \ell \leq J}}\int_{T}^{2T}|\mathcal{M}_{T_\ell}(1+2it)|^2\exp\left(\frac{2\widehat{k}}{\theta_{j-1}}\right)\left(\frac{|\mathcal{P}_{j,T_\ell}(\tfrac12+it)|}{K_j}\right)^{2s_j}\prod_{i=1}^{j-1} |\mathcal{N}_{i,T_\ell}(\tfrac12+it)|^2\dt \\ 
			\ll &T(\log T)^{k_1^2+\dots+k_r^2}.
		\end{align*}
	\end{prop}
	
	\begin{proof}
		We argue similarly to the proof of Proposition \ref{goodSetLemma}. Since the length of the the product of all the Dirichlet polynomials in the integral has length $\leq T^{8/10}$, we are again in a position to apply Lemma \ref{coprimeDirichletPolysLemma}. The same computations as in Proposition \ref{goodSetLemma}, in addition to applying Lemma \ref{soundHighMomentsLemma}, yields 
		\begin{align*}
			&\frac{1}{T}\int_{T}^{2T}|\mathcal{M}_{T_\ell}(1+2it)|^2\exp\left(\frac{2\widehat{k}}{\theta_{j-1}}\right)\left(\frac{|\mathcal{P}_{j,T_\ell}(\tfrac12+it)|}{K_j}\right)^{2s_j}\prod_{i=1}^{j-1} |\mathcal{N}_{i,T_\ell}(\tfrac12+it)|^2\dt \\
			\ll &(\log T_1)^{k_1^2+\dots+k_r^2}\exp\left(\frac{2\widehat{k}}{\theta_{j-1}}\right)\left(K_j^{-2s_j}s_j!\left(\sum_{T_{j-1}<p\leq T_j} \frac{|\mathbf{\Lambda}_{T_\ell}(p)|^2}{p}\right)^{s_j}\right)\left(\frac{\log T_{j-1}}{\log T_{1}}\right)^{k_1^2+\dots+k_r^2} \\
			\ll &(\log T_{j-1})^{k_1^2+\dots+k_r^2}\exp\left(\frac{2\widehat{k}}{\theta_{j-1}}\right)K_j^{-2s_j}s_j!\left(2(k_1^2+\dots+k_r^2) \right)^{s_j},
		\end{align*}
		where we used Selberg orthogonality in the end, in addition to the observation $$\log \frac{\log T_{j}}{\log T_{j-1}} \leq 2(k_1^2+\dots+k_r^2)$$ for large enough $T$. By Stirling's formula we have 
		\begin{align*}
			&\log\left(\frac{1}{\theta_j}\right)(\log T_{j-1})^{k_1^2+\dots+k_r^2}\exp\left(\frac{2\widehat{k}}{\theta_{j-1}}\right)K_j^{-2s_j}s_j!\left(2(k_1^2+\dots+k_r^2) \right)^{s_j} \\
			\ll &(\log T)^{k_1^2+\dots+k_r^2} \exp \left(\frac{2\widehat{k}}{\theta_{j-1}} -\frac{1}{20\theta_j}\log\left(\frac{1}{\theta_j}\right)\right) \exp \left(\frac{1}{10\theta_j}\log\left(\frac{2}{10e}\right)-\frac{1}{2}\log\theta_j+\log_2\left(\frac{1}{\theta_j}\right)\right) \\
			\ll &(\log T)^{k_1^2+\dots+k_r^2}\exp\left(-\frac{1}{\theta_j}\log\left(\frac{1}{\theta_j}\right)\left(\frac{2\widehat{k}e}{\log \theta_j}+\frac{1}{20}\right)\right) \ll (\log T)^{k_1^2+\dots+k_r^2}\exp \left(-\frac{C}{\theta_j}\log\left(\frac{1}{\theta_j}\right)\right)
		\end{align*}
		for a constant $C$ which is bigger than $0$ because $$\frac{1}{20}+\frac{2\widehat{k}e}{\log(\theta_j)}\geq \frac{1}{20}-\frac{2\widehat{k}e}{1000\widehat{k}}>0.$$ Using $J-j \ll \log(1/\theta_j)$ we conclude that
		\begin{align*}
			&\sum_{\substack{1 \leq j \leq J, \\ j\leq \ell \leq J}}\int_{T}^{2T}|\mathcal{M}_{T_\ell}(1+2it)|^2\exp\left(\frac{2\widehat{k}}{\theta_{j-1}}\right)\left(\frac{|\mathcal{P}_{j,T_\ell}(\tfrac12+it)|}{K_j}\right)^{2s_j}\prod_{i=1}^{j-1} |\mathcal{N}_{i,T_\ell}(\tfrac12+it)|^2\dt \\
			\ll &(\log T)^{k_1^2+\dots+k_r^2}\sum_{1 \leq j \leq J} \exp \left(-\frac{C}{\theta_j}\log\left(\frac{1}{\theta_j}\right)\right) \ll (\log T)^{k_1^2+\dots+k_r^2},
		\end{align*}
		where the last inequality is true because the last sum is bounded.
	\end{proof}
	
	\section{Deducing Theorem \ref{largeDeviationsCorollary} and Corollary \ref{hurwitzCorollary} from Theorem \ref{mainThmLowerBound} and \ref{mainThmUpperBound}}
	{Throughout this section, we shall use the following shorthand notation for our distributional function: $$\Phi_T(V_1,\dots,V_r) \coloneq \textup{meas}\left( t \in [1,T] : \frac{\log|L(\tfrac12+it,\pi_j)|}{\sqrt{\frac12 \log \log T}} \geq V_j, \enspace \forall 1 \leq j \leq r\right).$$ Throughout we assume $V_j \asymp \sqrt{\log\log T}$, so we fix positive constants $C_j$ straight away so that $$V_j=C_j\sqrt{\log\log T}.$$ We prove two lemmata before we start our proof of Theorem \ref{largeDeviationsCorollary}.
		
		\begin{lem}\label{largeDeviationsLemma1}
			We have
			\begin{align}\label{distributionalEquality}
				&\int_{1}^{T} \prod_{j=1}^r |L(\tfrac12+it,\pi_j)|^{2k_j} \dt \\
				= &2^r\left(\frac12 \log \log T\right)^{r/2} k_1\cdots k_r \int_{-\infty}^\infty \cdots \int_{-\infty}^{\infty} e^{2\sqrt{\frac12 \log\log T}(k_1W_1+\dots+k_rW_r)}\Phi_T(W_1,\dots,W_r)\textup{d}W_1\cdots \textup{d}W_r. \nonumber
			\end{align}
		\end{lem}
		
		\begin{proof}
			We start by observing that $$|L(\tfrac12+it,\pi_j)|^{2k_j} = \int_0^\infty \mathbf{1}_{y\leq |L(\tfrac12+it,\pi_j)|^{2k_j}} \, \textup{d}y,$$ where $\mathbf{1}$ denotes the usual indicator function. Using this simple reformulation and Fubini's theorem several times we thus see that 
			\begin{align*}
				&\int_{1}^{T} \prod_{j=1}^r |L(\tfrac12+it,\pi_j)|^{2k_j} \dt \\
				= &\int_{1}^{T} \int_{0}^\infty \cdots \int_0^\infty \prod_{j=1}^r \mathbf{1}_{y_j \leq |L(\tfrac12+it,\pi_j)|^{2k_j}} \, \text{d}y_1\text{d}y_2\cdots\text{d}y_r\text{d}t \\
				= &\int_{0}^\infty \cdots \int_0^\infty\int_{1}^{T}\prod_{j=1}^r \mathbf{1}_{y_j \leq |L(\tfrac12+it,\pi_j)|^{2k_j}} \, \text{d}t\text{d}y_1\text{d}y_2\cdots\text{d}y_r \\
				= &\int_{0}^{\infty} \cdots \int_0^\infty \textup{meas}\left( t \in [1,T] : |L(\tfrac12+it,\pi_j)|^{2k_j} \geq y_j, \enspace \forall 1 \leq j \leq r\right) \, \text{d}y_1\text{d}y_2\cdots\text{d}y_r.
			\end{align*}
			Now we substitute $y_j=\exp(2k_jW_j\sqrt{\frac12\log\log T})$ for all $1\leq j \leq r$. Then the previous integral equals 
			$$2^r\left(\frac12 \log\log T\right)^{r/2}k_1\cdots k_r\int_{-\infty}^\infty \cdots \int_{-\infty}^\infty e^{2\sqrt{\frac12\log\log T}(k_1W_1+\dots+k_rW_r)}\Phi_T(W_1,\dots,W_r)\,\text{d}W_1\cdots\text{d}W_r.$$
		\end{proof}
		
		\begin{lem}\label{largeDeviationsIntegralMassLemma}
			Let $\varepsilon=o(1)$ be a given positive real number. Let furthermore $S_2,\dots,S_r$ be any subsets of $\mathbb{R}$, and define $I_1^+ \coloneq (V_1(1+\varepsilon),\infty)$ and $I_1^{-} \coloneq (-\infty,V_1(1-\varepsilon))$. Finally, let $k_j=C_j/\sqrt{2}$. Then
			\begin{align*}
				&\int_{I_1^{\pm}} \int_{S_2}\cdots \int_{S_r} e^{2\sqrt{\frac12 \log\log T}(k_1W_1+\dots+k_rW_r)}\Phi_T(W_1,\dots,W_r)\textup{d}W_1\cdots \textup{d}W_r \\
				\ll\, &\frac{T}{(\log \log T)^{r/2}}\exp\left(\frac{V_1^2}{2}(1-\varepsilon^2)+\frac{V_2^2}{2}+\dots+\frac{V_r^2}{2}\right).
			\end{align*}
		\end{lem}
		
		\begin{proof}
			Let us first consider the case $I_1^+$. For any $\delta>0$ and $W_1 \in I_1^+$ we have $$e^{2\sqrt{\frac12 \log\log T}k_1W_1}=e^{2\sqrt{\frac12 \log\log T}k_1W_1(1+\delta-\delta)}\leq e^{-2\sqrt{\frac12\log\log T}k_1V_1(1+\varepsilon)\delta}e^{2\sqrt{\frac12\log\log T}k_1W_1(1+\delta)}.$$ Choosing $\delta=\varepsilon$, we thus have
			\begin{align*}
				&\int_{I_1^{+}} \int_{S_2}\cdots \int_{S_r} e^{2\sqrt{\frac12 \log\log T}(k_1W_1+\dots+k_rW_r)}\Phi_T(W_1,\dots,W_r)\textup{d}W_1\cdots \textup{d}W_r \\
				\leq &e^{-2k_1\varepsilon V_1(1+\varepsilon)\sqrt{\frac12 \log\log T}}\int_{-\infty}^\infty \cdots \int_{-\infty}^\infty e^{2\sqrt{\frac12 \log\log T}((1+\delta)k_1W_1+\dots+k_rW_r)}\Phi_T(W_1,\dots,W_r)\textup{d}W_1\cdots \textup{d}W_r.
			\end{align*}
			Recalling we have $k_j=\frac{C_j}{\sqrt{2}}$, Lemma \ref{largeDeviationsLemma1} and Theorem \ref{mainThmUpperBound} implies
			\begin{align*}
				&\int_{-\infty}^\infty \cdots \int_{-\infty}^\infty e^{2\sqrt{\frac12 \log\log T}((1+\delta)k_1W_1+\dots+k_rW_r)}\Phi_T(W_1,\dots,W_r)\textup{d}W_1\cdots \textup{d}W_r \\
				\ll\, &\frac{T}{(\log\log T)^{r/2}}\exp\left(\frac{V_1^2(1+\varepsilon)^2}{2}+\frac{V_2^2}{2}+\dots+\frac{V_r^2}{2}\right).
			\end{align*} 
			Thus, we can conclude that 
			\begin{align*}
				&\int_{I_1^{+}} \int_{S_2}\cdots \int_{S_r} e^{2\sqrt{\frac12 \log\log T}(k_1W_1+\dots+k_rW_r)}\Phi_T(W_1,\dots,W_r)\textup{d}W_1\cdots \textup{d}W_r \\
				\ll &\frac{T}{(\log \log T)^{r/2}}\exp\left(\frac{V_1^2}{2}(1-\varepsilon^2)+\frac{V_2^2}{2}+\dots+\frac{V_r^2}{2}\right). 
			\end{align*}

			Now, we turn to the case of $I_1^{-}$. This is very similar, expect that we now have for any $\delta>0$ and any $W_1 \in I_1^{-}$, that $$e^{2\sqrt{\frac{1}{2}\log\log T}k_1W_1} \leq e^{2\sqrt{\frac12\log\log T}k_1V_1(1-\varepsilon)\delta}e^{2\sqrt{\frac12\log\log T}k_1W_1(1-\delta)}.$$ The next steps are identical to the previous case.
		\end{proof}
	}
	\begin{proof}[Proof of Theorem \ref{largeDeviationsCorollary}]
		{By Chebyshev's inequality and Theorem \ref{mainThmUpperBound}, we have for any fixed $k_j>0$ that
			\begin{align*}
				\Phi_T(V_1,\dots,V_r) &= \textup{meas}\left( t \in [1,T] : |L(\tfrac12+it,\pi_j)| \geq (\log T)^{\sqrt{\frac{C_j^2}{2}}}, \enspace \forall 1 \leq j \leq r\right) \\
				&\leq (\log T)^{-\sum_{1\leq j \leq r} 2k_j\sqrt{\frac{C_j^2}{2}}}\int_{1}^{T} |L(\tfrac12+it,\pi_1)|^{2k_1}\cdots |L(\tfrac12+it,\pi_r)|^{2k_r} \, \textup{d}t \\
				&\ll T(\log T)^{\sum_{1\leq j \leq r} k_j^2-2k_j\sqrt{\frac{C_j^2}{2}}}.
			\end{align*}
			Choosing $k_j = \frac{C_j}{\sqrt{2}}$ we see that the above is indeed $$\ll T\exp\left(-\frac{V_1^2+\dots+V_r^2}{2}\right).$$
			
			Now we turn to the lower bound, where we have to work more directly with the integral on the right hand side of (\ref{distributionalEquality}). Our goal is to show that the integrand has its mass in $(V_1(1-\varepsilon),V_1(1+\varepsilon))\times \cdots \times (V_r(1-\varepsilon),V_r(1+\varepsilon))$. To this end, we split up the integral in each variable $W_j$ into $(-\infty, V_j(1-\varepsilon)), [V_j(1-\varepsilon),V_j(1+\varepsilon)]$ and $(V_j(1+\varepsilon),\infty)$. We will denote these intervals by $I_j^-, I_j$ and $I_j^+$ respectively. Doing this in all variables we split up the integral into $3^r$ boxes in $\mathbb{R}^r$. If all of the sides of the box are of the form $\mathcal{I}_j$, we leave the integral as it is. If this is not the case, it means that at least one of the sides in the box take the form $\mathcal{I}_j^{-}$ or $\mathcal{I}_j^{+}$. Without loss of generality, say $j=1$ (up to a possible permutation). Let us name the other sides of the box $S_2,\dots,S_r$. Thus, if $k_j=C_j/\sqrt{2}$, Lemma \ref{largeDeviationsIntegralMassLemma} implies 
			\begin{align*}
				&\int_{-\infty}^\infty \cdots \int_{-\infty}^\infty e^{2\sqrt{\frac12\log\log T}(k_1W_1+\dots+k_rW_r)}\Phi_T(W_1,\dots,W_r)\,\text{d}W_1\cdots\text{d}W_r \\
				=\, &\int_{I_1}\cdots \int_{I_r} e^{2\sqrt{\frac12\log\log T}(k_1W_1+\dots+k_rW_r)}\Phi_T(W_1,\dots,W_r)\,\text{d}W_1\cdots\text{d}W_r \\ 
				&+O\left(\frac{T}{(\log \log T)^{r/2}}\sum_{j=1}^r \exp\left(\frac{V_1^2}{2}+\dots+\frac{V_j^2}{2}(1-\varepsilon^2)+\dots+\frac{V_r^2}{2}\right)\right) \\
				\ll\, &\varepsilon^r V_1\cdots V_r \Phi_T(V_1(1-\varepsilon),\dots,V_r(1-\varepsilon))\exp\left(2\sqrt{\frac12\log\log T}(1+\varepsilon)(k_1V_1+\dots+k_rV_r)\right)\\ 
				&+ O\left(\frac{T}{(\log \log T)^{r/2}}\sum_{j=1}^r \exp\left(\frac{V_1^2}{2}+\dots+\frac{V_j^2}{2}(1+\varepsilon^2)+\dots+\frac{V_r^2}{2}\right)\right)
			\end{align*} 
			Tidying up and using Lemma \ref{largeDeviationsLemma1}, we have
			\begin{align*}
				&\frac{1}{T}\Phi_T(V_1(1-\varepsilon),\dots,V_r(1-\varepsilon)) \\ 
				&+ O\left(\frac{\sum_{j=1}^r \exp\left(-\frac{V_1^2}{2}(1+2\varepsilon)-\dots-\frac{V_j^2}{2}(1+2\varepsilon-\varepsilon^2)-\dots-\frac{V_r^2}{2}(1+2\varepsilon)\right)}{(\log \log T)^{r/2}\varepsilon^rV_1\cdots V_r}\right) \\
				\gg\, &\frac{1}{T}\frac{\exp\left(-(1+\varepsilon)\left(\frac{V_1^2}{2}+\dots+\frac{V_r^2}{2}\right)\right)}{(\log \log T)^{r/2}\varepsilon^rV_1\cdots V_r}\\ \enspace&\enspace \times \int_{-\infty}^\infty \cdots \int_{-\infty}^\infty e^{2\sqrt{\frac12\log\log T}(k_1W_1+\dots+k_rW_r)}\Phi_T(W_1,\dots,W_r)\,\text{d}W_1\cdots\text{d}W_r \\
				\gg \, &\frac{\exp\left(-(1+\varepsilon)\left(V_1^2+\dots+V_r^2\right)\right)}{(\log\log T)^{r}\varepsilon^rV_1\cdots V_r}(\log T)^{k_1^2+\dots+k_r^2} \\
				\gg &\exp\left(-(1+o(1))\left(\frac{V_1^2}{2}+\dots+\frac{V_r^2}{2}\right)\right),
			\end{align*}
			with $\varepsilon = \frac{D}{(\log \log T)^2}$ for some constant $D$. Comparing the error term on the left hand side to main term, we can thus conclude that $$\frac{1}{T}\Phi_T(V_1(1-\varepsilon),\dots,V_r(1-\varepsilon)) \gg \exp\left(-(1+o(1))\left(\frac{V_1^2}{2}+\dots+\frac{V_r^2}{2}\right)\right).$$ Letting $V_j\mapsto \frac{V_j}{1-\varepsilon}$, the desired inequality follows. 
		}
	\end{proof}

	\begin{proof}[Proof of Corollary \ref{hurwitzCorollary}]
		Theorem \ref{mainThmUpperBound} give us sharp upper bound for $2k$th moment for any Dirichlet $L$-function with primitive character, and hence also any imprimitive character, up to order. If $k>1/2$, we use Jensen's inequality and (\ref{hurwitzZetaDirichletLinearCombs}) to see that $$|\zeta(s,\alpha)|^{2k} \ll \sum_{\chi} |L(s,\chi)|^{2k},$$  which gives us the desired upper bound. If $k\leq 1/2$, we instead apply the inequality $(x+y)^{1/p} \leq x^{1/p}+y^{1/p}$, which is valid for any $x,y\geq 0$, $p\geq 1$.
		
		{For the lower bound, we apply Hölder's inequality to the function $q^{-it}\zeta(\tfrac12+it,\alpha)\overline{\zeta(\tfrac12+it)}|\zeta(\tfrac{1}{2}+it)|^{2(k-1)}$ in the following way:
			\begin{align*}
				&\frac{1}{T}\left|\int_{1}^{T} q^{-it}\zeta(\tfrac12+it,\alpha)\overline{\zeta(\tfrac12+it)}|\zeta(\tfrac{1}{2}+it)|^{2(k-1)}\dt \right| \\
				\leq &\left(\frac{1}{T}\int_{1}^{T}|\zeta(\tfrac12+it,\alpha)|^{2k}\dt\right)^{1/(2k)}\left(\frac{1}{T} \int_{1}^{T} |\zeta(\tfrac12+it)|^{2k}\dt\right)^{(2k-1)/(2k)}
			\end{align*}
			By the sharp moment bound for the Riemann zeta function \cite{Harper1}, we conclude from the above that $$\frac{1}{T}\int_{1}^{T}|\zeta(\tfrac12+it,\alpha)|^{2k}\dt \geq \frac{1}{T}\left|\int_{1}^{T} q^{-it}\zeta(\tfrac12+it,\alpha)\overline{\zeta(\tfrac12+it)}|\zeta(\tfrac{1}{2}+it)|^{2(k-1)}\dt \right|^{2k}(\log T)^{k^2(1-2k)}.$$
			If we can show that the integral on the right hand side is $\asymp T(\log T)^{k^2}$, then the desired conclusion follows. This is precisely what we expect as well, because $q^{-it}\overline{\zeta(\tfrac12+it)}|\zeta(\tfrac{1}{2}+it)|^{2(k-1)}$ should ``pick out'' the term $L(\tfrac12+it,\chi_0)\approx\zeta(\tfrac12+it)$ from the sum in (\ref{hurwitzZetaDirichletLinearCombs}), and by independence of we expect the other moments to be smaller. Indeed, opening up the integral we see
			\begin{align*}
				&\int_{1}^{T} q^{-it}\zeta(\tfrac12+it,\alpha)\overline{\zeta(\tfrac12+it)}|\zeta(\tfrac{1}{2}+it)|^{2(k-1)}\dt \\
				=\,&\frac{q^{1/2}}{\varphi(q)}\sum_{\chi} \overline{\chi(a)}\int_{1}^{T}L(\tfrac12+it,\chi)\overline{\zeta(\tfrac12+it)}|\zeta(\tfrac{1}{2}+it)|^{2(k-1)} \dt \\
				\asymp\, &T(\log T)^{k^2} + O\left(T(\log T)^{k^2-k+\frac{1}{2}}\right).
			\end{align*}
			Here we have used Theorem \ref{mainThmLowerBound} and \ref{mainThmUpperBound} to compute the moments. Because $k\geq \frac{1}{2}$, this finishes the proof.
		}
	\end{proof}

\end{document}